\documentclass[11pt]{article}
\usepackage[tbtags]{amsmath}
\usepackage{amssymb}
\usepackage{amsthm}
\usepackage[misc]{ifsym}
\usepackage{cases}
\usepackage{cite}
\usepackage{mathrsfs}
\usepackage{xcolor}
\usepackage{graphicx}
\usepackage{float}
\usepackage{color}
\usepackage{hyperref}
\usepackage{amsfonts}

\hypersetup{hypertex=true,
	colorlinks=true,
	linkcolor=red,
	anchorcolor=blue,
	citecolor=blue}

\numberwithin{equation}{section}
\normalsize

\title{\bf The Optimal Control Problem of Stochastic Differential System with Extended Mixed Delays and Applications
	\thanks{This work is financially supported by the National Natural Science Foundations of China (12471419, 12271304), and the Shandong Provincial Natural Science Foundation (ZR2024ZD35).}}
\author{\normalsize
	Xinpo Li\thanks{\it Research Center for Mathematics and Interdisciplinary Sciences, Shandong University, Qingdao 266237, P.R. China; and School of Mathematics, Shandong University, Jinan 250100, P.R. China, E-mail: lixinpo@mail.sdu.edu.cn},
	\ Jingtao Shi\thanks{\it Corresponding author, School of Mathematics, Shandong University, Jinan 250100, P.R. China, E-mail: shijingtao@sdu.edu.cn}}
\date{}

\newtheorem{mypro}{Proposition}[section]
\newtheorem{mythm}{Theorem}[section]

\newtheorem{mylem}{Lemma}[section]
\newtheorem{Remark}{Remark}[section]

\begin{document}
	
	\maketitle
	
	\noindent{\bf Abstract:}\quad
	This paper investigates an optimal control problem where the system is described by a stochastic differential equation with extended mixed delays that contain point delay, extended distributed delay, and extended noisy memory. The model is general in that the extended mixed delays of the state variable and control variable are components of all the coefficients, in particular, the diffusion term and the terminal cost. To address the difficulties induced by the extended noisy memory, by stochastic Fubini theorem, we transform the delay variational equation into a Volterra integral equation without delay, and then a kind of backward stochastic Volterra integral equation with Malliavin derivatives is introduced by the developed coefficient decomposition method and the generalized duality principle. Therefore, the stochastic maximum principle and the verification theorem are established. Subsequently, with Clark-Ocone formula, the adjoint equation is expressed as a set of anticipated backward stochastic differential equations. Finally, a nonzero-sum stochastic differential game with extended mixed delays and a linear-quadratic solvable example are discussed, as applications.
	
	\vspace{2mm}
	
	
	\vspace{2mm}
	
	
	\section{Introduction}
	
	Stochastic control systems have demonstrated remarkable effectiveness and accuracy in modeling complex phenomena arising in economic finance, public health, social management, information networks, aerospace, artificial intelligence and so on. Consequently, the study of optimal control problems for stochastic systems has consistently garnered significant and sustained interest from the academic community over the past several decades. In 1965, Kusher \cite{Kusher65} pioneered the examination of stochastic optimal control problems where the diffusion term is independent of both state variables and control variables. This issue subsequently triggered a sequence of intensive research developments. Specifically, Bismut \cite{Bimust78} and Bensoussan \cite{Bensoussan81} subsequently established stochastic maximum principles by introducing adjoint processes and applying Girsanov transformation, respectively. However, their results relied on restrictive assumptions, the convexity of the control domain or the degeneracy of the diffusion coefficient when control domain is non-convex. Until 1990, Peng \cite{PSG90} innovatively proposed the second-order variational method and introduced the \emph{backward stochastic differential equations} (BSDEs, in short) as the adjoint equation to present the global maximum principle, in which the theory of nonlinear BSDEs was established in \cite{PP90}. Yong and Zhou \cite{YZ99} systematically reviewed these context.
	
	The systems investigated in the aforementioned studies depend solely on the current state variable and control variable. Nonetheless, in reality, such as signal transmission in communication networks as well as the propagation of light and sound in physical systems, time delay effects are inherently present. In mathematics, this phenomenon are referred to as historical dependency and are commonly modeled by \emph{stochastic delay differential equations} (SDDEs, in short). Mohammed \cite{Mohammed84,Mohammed98} and Mao\cite{MXR08} have conducted a comprehensive and methodical review of the subject for interested readers. Owing to the prevalence of delay effects in practical systems, the examination of stochastic optimal control for delayed systems is of considerable practical significance and profound theoretical value.   
	
	In the meanwhile, there are some difficulties in studying the optimal control problem of delay systems, not only because there is a lack of It\^o's formula to deal with delay terms so far, but also because the \emph{dynamic programming principle} (DPP, in short) of general delay problems essentially gives rise to an infinite dimensional \emph{Hamilton-Jacobi-Bellman} (HJB, in short) equation, which poses significant challenges in its solution.
	
	Alternatively, stochastic maximum principle provides a framework for treating optimal control problems in stochastic delay systems. \O ksendal and Sulem \cite{OS00} investigated a type of stochastic delay systems that contained the average value of the previous duration, referred to as the \emph{distributed delay}, in addition to the present value. To overcome the infinite dimensional difficulties induced by the delay term, a zero value \emph{backward ordinary differential equation} (BODE, in short) was introduced as part of adjoint equation based on its special structure, while the other part is a BSDE in this research, and then the sufficient maximum principle was derived. In 2010, Chen and Wu \cite{CW10} firstly gained stochastic maximum principle for system with the point delays of state variables and control variables by introducing \emph{anticipated backward stochastic differential equations} (ABSDEs, in short) of Peng and Yang \cite{PY09}'s type, as the adjoint equation, where a duality principle between ABSDEs and SDDEs was developed. In 2021, Meng and Shi \cite{MS21} introduced three ABSDEs as adjoint equation through the second-order Taylor expansion, resulting in the general maximum principle that stochastic control systems contained the point delays of state variables and control variables. Please review the literature \cite{ZLX20, WW15, LWW20, WW17, YZY12, MS15, ZF21, YY25, AHOP13, CY15, DNW23, WZLF24, WZX21, ZX17} and the references therein for further relevant developments about point delay and distributed delay. Moreover, it is frequently utilized as the primary tool for obtaining the stochastic maximum principle of both point delay and distributed delay systems.
	\O ksendal and Agram \cite{AO19-2,AO19-3}, Feng et al. \cite{FGWX24} introduced the ABSDEs with dual operator to overcome the challenges posed by \emph{elephant memory}. Han and Li \cite{HL24} studied the stochastic optimal control problem for control systems with \emph{time-varying delay} by introducing a type of ABSDEs. Guo et al. \cite{GXZ24} studied the stochastic maximum principle for a \emph{generalized mean-field delay} control problem, in which the state equation depended on the distribution. Moreover, they established the existence and uniqueness of solutions to the adjoint equation, formulated as a \emph{mean-field anticipated backward stochastic differential equations} (MFABSDEs), which all the derivatives of the coefficients were in Lions’ sense. Guatteri et al. \cite{GMW25} studied the global maximum principle to the case of stochastic delay differential equations of mean-field type. More precisely, the coefficients of stochastic control problem depended on the state, on the past trajectory and on its expected value. Moreover, the control entered the noise coefficient and the control domain could be non-convex. The approach was based on a lifting of the state equation to an infinite dimensional Hilbert space that removed the explicit delay in the state equation. The main ingredient in the proof of the maximum principle was a precise asymptotic for the expectation of the first order variational process, which allowed them to neglect the corresponding
	second order terms in the expansion of the cost functional. Zhang \cite{ZQX25} regarded the delay term of the state as a control process and then the Ekeland’s variational principle was applied to obtain a general maximum principle for the optimal control problems. Meng et al. \cite{MSWZ25} obtained a general maximum principle for a stochastic optimal control problem where the control domain was an arbitrary nonempty set and all the coefficients (especially the diffusion term and the terminal cost) contained the control and state delay. In order to overcome the difficulty of dealing with the cross term of state and its delay in the variational inequality, they proposed a new method: transform a delayed variational equation into a Volterra integral equation without delay, and introduced novel first-order, second-order adjoint equations via the \emph{backward stochastic Volterra integral equations} (BSVIEs, in short) theory. Finally, they expressed these two kinds of adjoint equations in more compact anticipated backward stochastic differential equation types for several special yet typical control systems.
	
	Forward Volterra integral equations were first introduced by Volterra \cite{VV59}. Since then, extensive studies have been devoted to optimal control problems for forward Volterra integral systems. In contrast, relatively little attention has been paid to the optimal control of \emph{forward stochastic Volterra integral euqations} (FSVIEs, in short). One possible reason is that the theory of BSVIEs was not established until 2002, when Lin \cite{LJZ02} developed the framework of Type-I BSVIEs. Subsequently, Yong \cite{YJM06} introduced Type-II BSVIEs in 2006 and derived the first stochastic maximum principle for optimal control problems of forward stochastic Volterra integral systems with convex control domains. More recently, Wang and Yong \cite{WY23} proposed an auxiliary process approach and obtained a general maximum principle allowing for nonconvex control domains. Additional related works can be found in~\cite{HY23,WTX20, WZ17}.
	
	In the aforementioned studies, point delay and distributed delay are the two primary delay types taken into account. While these two types of delays possess some theoretical relevance and economic interpretability in describing the system's historical reliance, they are generally insufficient to capture the influence of random uncertainty elements present in the real dynamical system. For instance, in the financial market environment, the dynamic process exhibits notable non-stationarity and randomness due to monetary policy adjustments, business and investment decisions made by enterprises, investor behavioral deviations, information asymmetry, and the occurrence of abrupt \emph{black swan events}, which are not only affect investors’ asset positions indirectly through indicators like volatility and rate of return, but also they directly influence how stock values fluctuate and change over time. Moreover, stock price fluctuations are influenced by the strategic behaviors adopted by multiple market participants. Investors, trading institutions, and other agents dynamically adjust their decisions not only based on current market information but also considering the decisions of other participants. Consequently, the formation and evolution of stock prices are not governed by a single decision-maker; rather, they emerge from the joint effects of stochastic factors and strategic interactions among multiple agents. In addition, the goals of different participants in the stock market are different, such as profit maximization, risk management, and the maintenance of market stability. This feature of interdependence and inconsistent goals can naturally transform the investment decision-making process in the stock market into a nonzero-sum game model. Within this framework, participants select strategies to maximize or minimize their respective objective functions, thereby capturing the essential features of competitive and cooperative behaviors observed in real financial markets.
	
	From the perspective of nonzero-sum stochastic differential games, the tendency of the stock price is now a dynamic process of interaction between numerous random factors and strategies rather than the product of a single subject's control. With considerable randomness, dependence, and structural complexity, its state variables typically display the form of time-evolving stochastic processes. Thus, the systematic characterization of the statistical properties, temporal dependence structure, and dynamic evolution of stock price processes, as well as the development of effective prediction methodologies based thereon, have become a central issue in both theoretical investigations and practical applications.
	
	Against this background, time series analysis, as a fundamental mathematical tool for studying stochastic dynamical systems, has been widely employed to describe the mechanism of action of uncertain factors in the financial market at different time scales, to reveal the intrinsic structure of stock price games, and to provide a solid theoretical foundation for the design of decision-making strategies. Among various time series models, the \emph{autoregressive moving average} (ARMA, in short) model occupies a prominent position in financial modeling, owing to its ability to simultaneously capture the memory effects of asset dynamics and the influence of random disturbances (for example, Peter and Richard \cite{PR02}). More specifically, the ARMA model characterizes the dependence of the current state on past information, which is referred to as memory effects, by representing the present value as a finite linear combination of past states and a finite linear combination of past random disturbances. Its general form is given by 
	\begin{equation}\label{ARMA}
		X_t=\sum_{j=1}^{p}a_jX_{t-j}+\sum_{j=0}^{q}b_j\epsilon_{t-j},\ t\in\mathbb{Z},
	\end{equation}
	where $\{\epsilon_t\}$ is WN$(0,\sigma^2)$, $a_j, b_j\in\mathbb{R}$ coming from the polynomials $A(z)=1-\sum_{j=1}^pa_jz^j\ne 0,\ |z|\le1 $ with $B(z)=1-\sum_{j=0}^qb_jz^j\ne 0,\ |z|\le1$, respectively. In addition, polynomials $A(z)=0$ and $B(z)=0$ 
	with real coefficients have no common roots with $b_0=1,\ a_pb_q\ne 0$.
	The autoregressive coefficient $a_j$ describes the persistent memory of the sequence to the historical state, while the moving average coefficient $b_j$ represents the lagged memory of the random disturbance to the current state. In time series analysis, these two kinds of coefficients are crucial. The former reflects the intrinsic memory structure of the state evolution and is crucial for examining the model's fitting capability, reversibility, and stability of judgment estimation. The latter, on the other hand, captures the temporal propagation mechanism of stochastic shocks and is primarily used to evaluate the effectiveness of estimation procedures, to guide model order selection, and to conduct statistical hypothesis testing on model parameters.
	
	The summation terms in the ARMA model is essentially a description of the weighted accumulation of historical information within a discrete time framework, which is effective in capturing and analyzing financial time series driven by long-term, low-frequency trading behaviors. However, there are a large number of short-term high-frequency trading behaviours in the actual financial market, which show more complex and rapid characteristics in terms of time scale and information response speed. It may be difficult to fully characterise its internal mechanism by relying only on the above discrete time weighted accumulation structure. Nevertheless, the structural form of the ARMA model and its inherent memory property naturally inspire us to introduce the following integral form of the memory term. Based on the cumulative summation structure characteristics of the autoregressive terms, we naturally extend this formulation to integral type delays as the following
	\begin{equation}\label{extended distributed delay}
		\int_0^t \phi(t,s)x(s)ds,
	\end{equation}
	which is named \emph{extended distributed delays} in \cite{MWZ25}, where kernel  $\phi(t,s),\{(t,s)\in\mathbb{R}^2,0\le s\le t\}$ represents the weight of the influence exerted by the historical state $x(s)$ prior to time $t$ on the system’s evolution at the current time. We refer to $\phi(t,s)$ as an \emph{autoregressive kernel} in view of its close structural analogy to the autoregressive coefficients in ARMA models, which serves to characterize the dependence of the system state on its entire past trajectory in continuous time.
	
	On the other hand, the moving average term in the ARMA model captures the persistent impact of white noise memory on the current state by means of a linear weighting of a finite number of random disturbances. In analogy with the autoregressive term, it is naturally extended to the following stochastic integral term with the kernel function in continuous time
	\begin{equation}\label{extended noisy memory}
		\int_0^t \psi(t,s)x(s)dW(s).
	\end{equation}
	Inspired by \cite{DMOR16}, we denote this term by \emph{extended noisy memory}, where $\psi(t,s),\{(t,s)\in\mathbb{R}^2,0\le s\le t\}$ represents the degree of influence of the historical trajectory $x(s),0\le s\le t$ affected by noise memory prior to time $t$ on the current state, which is the natural extension of the moving average coefficient $b_j$ of the ARMA model in continuous time, and is therefore referred to as the \emph{moving average kernel}.
	
	Furthermore, the dynamic evolution of asset prices (system state) in the investment and asset allocation problem is influenced not only by extended distributed delays and the extended noisy memory, but also the investment strategy (control) of the investors, which has a more direct and substantial impact on the asset changes process. It is challenging to accurately represent the dynamic nature of assets in the context of investment solely from the memory structure of asset prices because investment strategies in the real financial market frequently include continuity, delay effects, and cumulative effects. Consequently, there exists difficult to effectively reflect the dynamic characteristics of assets in the context of investment only from the memory structure of asset prices. It is therefore necessary to incorporate delay effects in the control variables into the modeling framework, so as to characterize the persistent influence of past investment decisions on the current asset state. By analogy with the extension process for extended distributed delays as before, we introduce
	\begin{equation}\label{extended intrgeal control}
		\int_0^t \kappa(t,s)u(s)ds,
	\end{equation}
	where $u(\cdot)$ denotes the investor’s investment strategy (control), while the kernel $\kappa(t,s),\{(t,s)\in \mathbb{R}^2 , 0\le s\le t\}$ reflects the persistent influence of past control actions on the current state of the system.
	
	However, given that random elements like execution failures, liquidity shocks, and information asymmetry frequently accompany control in the real financial market, the control decisions of investors may also exert a persistent stochastic influence on the system through noise. Inspired by the extended noisy memory, this work introduces term as following
	\begin{equation}\label{extended noisy control}
		\int_0^t \eta(t,s)u(s)dW(s).
	\end{equation}
	This model is proposed for the first time to the best of our knowledge, in which the kernel $\eta(t,s), \{(t,s)\in\mathbb{R}^2, 0\le s\le t\}$ has the same structural properties as the moving average kernel. The stochastic maximum principle for systems with noisy memory was first investigated by Malliavin calculus in \cite{DMOR16}. Subsequent related developments were carried out in \cite{CZ26, KWX25, ZF22, Dahl20, ZQX25-2, MJS23, DFT22, ML19}. However, these works did not consider systems incorporating extended distributed delays together and extended noisy memory with (\ref{extended intrgeal control}), (\ref{extended noisy control}). In light of the preceding discussion, in order to give the optimal strategy of investors in a stock framework within financial markets, which naturally leads us to investigate the optimal control of stochastic delay differential systems that incorporate the extended distributed delays and extended noisy memory of the state variable with control variable.
	
	Assume that $(\Omega,\mathcal{F},\mathbb{F},\mathbb{P})$ is a complete filtered probability space with the filtration defined as $\mathbb{F}:=\{\mathcal{F}_t\}_{t\geq 0}=\sigma\{W(s);0 \leq s\leq t\}$, where $\{W(t)\}_{t\geq 0}$ is a one-dimensional 
	standard Brownian motion. Let $\delta>0$ is a given constant time delay parameter, with $U\subset\mathbb{R}^m$ is a nonempty convex set. For $t_0\ge0$, given a time duration $[t_0,T]$, we consider the following system with extended mixed delay
	\begin{equation}\left\{\begin{aligned}\label{state equation}
			dx(t)&=b\big(t,x(t),y(t),z(t),\kappa(t),u(t),\mu(t),\nu(t), \lambda(t)\big)dt\\
			&\quad+\sigma\big(t,x(t),y(t),z(t),\kappa(t),u(t),\mu(t),\nu(t), \lambda(t)\big)dW(t),\ t\in[t_0,T],\\
			x(t)&=\xi(t),\ u(t)=\varsigma(t),\ t\in[t_0-\delta,t_0],
		\end{aligned}\right.\end{equation}
	with the cost functional
	\begin{equation}\begin{aligned}\label{cost1}
			J(u(\cdot))=\mathbb{E}\bigg\{\int_{t_0}^{T}l\big(&t,x(t),y(t),z(t),\kappa(t),u(t),\mu(t),\nu(t),\lambda(t)\big)dt+h\big(x(T),y(T),z(T),\kappa(T)\big)\bigg\},
	\end{aligned}\end{equation}
	where $x(\cdot)\in\mathbb{R}^n,~u(\cdot)\in U$ represents the state variable and the control variable, respectively, and $y(t):=x(t-\delta),\ z(t):=\int_{t_0}^{t}\phi_1(t,s)x(s)ds,\ \kappa(t):=\int_{t_0}^{t}\psi_1(t,s)x(s)dW(s),\ \mu(t):=u(t-\delta),\ \nu(t):=\int_{t_0}^{t}\phi_2(t,s)u(s)ds,\ \lambda(t):=\int_{t_0}^{t}\psi_2(t,s)u(s)dW(s)$, with $\phi_1(\cdot,\cdot), \psi_1(\cdot,\cdot)\in\mathbb{R}^{n\times n},~\phi_2(\cdot,\cdot), \psi_2(\cdot,\cdot)\in\mathbb{R}^{m\times m}$. In the above, $b,\sigma,l,h$ are given coefficients with $b,\sigma: [0,T]\times\mathbb{R}^n\times\mathbb{R}^n\times\mathbb{R}^n\times\mathbb{R}^n\times U\times U\times U\times U\to\mathbb{R}^n,~l: [0,T]\times\mathbb{R}^n\times\mathbb{R}^n\times\mathbb{R}^n\times\mathbb{R}^n\times U\times U\times U\times U\to\mathbb{R},~h: \mathbb{R}^n\times\mathbb{R}^n\times\mathbb{R}^n\times\mathbb{R}^n \to\mathbb{R}$. Continuous function $\xi(\cdot)$ and square integrable function $\varsigma(\cdot)$ are the initial trajectories of the state and the control, respectively. 
	
	Define the admissible control set as follows
	\begin{equation}\begin{aligned}\label{admissible control set}
			\mathcal{U}_{ad}:=\Big\{u(\cdot):[t_0-\delta,T]\to\mathbb{R}^m \big|& u(\cdot)~ \text{is a}~
			U\text{-valued, square-integrable},\mathbb{F}\text{-adapted}\\
			&\text{process and}~u(t)=\varsigma(t),~t\in[t_0-\delta,t_0]\Big\}.  	
	\end{aligned}\end{equation}
	
	We state the optimal control problem as follows.
	
	\textbf{Problem (P)}: Our object is to find a control $u^\ast(\cdot)$ over $\mathcal{U}_{ad}$ such that (\ref{state equation}) is satisfied and (\ref{cost1}) is maximized, i.e.,
	\begin{equation*}
		J(u^\ast(\cdot))=\underset{u(\cdot)\in\mathcal{U}_{ad}}{\sup}J(u(\cdot)).
	\end{equation*}
	Any $u^\ast(\cdot)\in\mathcal{U}_{ad}$ that achieve the above supremum is called an \emph{optimal control} and the corresponding solution $x^\ast(\cdot)$ is called the \emph{optimal trajectory}.  
	The optimal pair is denoted as $(x^\ast(\cdot),u^\ast(\cdot))$.
	
	In what follows, we investigate \textbf{Problem (P)}. The main innovations and contributions of this paper can be summarized as follows.
	
	(1) The model is general. Inspired by the ARMA($p,q$) model in time series analysis, we study the extended noise memory (\ref{extended noisy memory}) which can cover the noise memory proposed in \cite{DMOR16} and provide an interpretation of extended distribute delay (\ref{extended distributed delay}), where the kernel functions $\phi(\cdot, \cdot)$ and $\psi(\cdot, \cdot)$ in (\ref{extended distributed delay}) and (\ref{extended noisy memory}) are termed as the autoregressive kernel and the moving average kernel, respectively. Together with point delay, these two types of delays constitute the extended mixed delay considered in this paper. The extended mixed delays of both the state and control variables are component of the coefficients of the stochastic control system, in particular the diffusion term and the terminal cost. It is worth mentioning that, to the best of our knowledge, the controlled extended noisy memory (\ref{extended noisy control}) is first proposed and studied in this paper.
	
	(2) Three methods are employed in this paper to overcome the difficulties induced by the extended noisy memory. Firstly, inspired by \cite{DMOR16}, the B-D-G inequality is applied to handle the challenges arising from the extended noisy memory, which enables us to derive a priori estimates for the solution and to introduce the variational equation. Secondly, with the help of the stochastic Fubini theorem, a novel transformation method is proposed to successfully convert the variational equation into SVIE without delay. Finally, the appearance of extended noisy memory results in the coefficients of SVIE failing to satisfy the conditions of Theorem 5.1 in \cite{YJM08}. To address this issue, we develop a \emph{coefficient decomposition} method and introduce a class of BSVIEs involving Malliavin derivatives, together with a generalized duality principle, as a part of the adjoint equations. Based on this framework, a stochastic maximum principle and a corresponding verification theorem are established.
	
	(3) The adjoint equation is expressed in a more compact form. With the help of the Clark-Ocone formula, the essential difficulty brought by Malliavin derivatives in the adjoint equation is solved, which allows the adjoint process to be expressed as a class of ABSDEs with Malliavin derivatives. Motivated by the heterogeneity of investment strategies and objectives among different market participants in the stock market, we further investigate a nonzero-sum stochastic differential game with extended mixed delays and a linear-quadratic (LQ, in short) example to demonstrate the simplicity and effectiveness of the aforementioned conclusions.
	
	The remainder of this paper is organized as follows. Section 2 introduces the fundamental notation and preliminary results. In Section 3, the variational equation with extended mixed delays is transformed into a Volterra integral equation without delay. Section 4 is devoted to the derivation of the stochastic maximum principle and the verification theorem, where the adjoint equation is constructed via coefficient decomposition and the generalized duality principle. In Section 5, the adjoint equation is linked with a class of ABSDEs. Applications to nonzero-sum stochastic differential game with extended mixed delays and LQ problems are presented in Section 6. Finally, concluding remarks are given in Section 7.
	
	Given finite time duration $T>0$,  $\mathbb{R}^{n\times m}$ represents the Euclidean space of all $n\times m$ real matrices, while $\mathbb{S}^n$ denotes the space containing all $n\times n$ symmetric matrices. $\mathbb{R}^{n\times m}$ is abbreviated as $\mathbb{R}^n$ when $m=1$. The symbols $|\cdot|$ and $\langle\cdot,\cdot\rangle$ indicate the norm in $\mathbb{R}^n$ and the inner product, respectively. The superscript $^\top$ frequently appears in the rest of content, denoting the transpose of a vector or matrix. $I$ represents the identity matrix of appropriate dimensions and $\mathbb{I}$ is indicator function. 	
	
	\section{Problem statement and preliminaries}
	
	Some preliminaries used in this work are presented in this section.
	
	Combining with the standard Brownian motion $\{W(t)\}_{t\geq 0}$ introduced in Section 1, under the probability measure $\mathbb{P}$, $\mathbb{E}[\cdot]$ designates the expectation, consequently, the conditional expectation is symbolized as $\mathbb{E}^{\mathcal{F}_t}[\cdot]:=\mathbb{E}[\cdot|\mathcal{F}_t]$.
	
	Then, we present the following spaces. For an integer $p>0$, we define
	\begin{equation*}\begin{aligned}
			&L^p([0,T];\mathbb{R}^n):=\bigg\{\mathbb{R}^n\mbox{-valued function }\phi(\cdot);\ \int_0^T|\phi(t)|^pdt<\infty\bigg\},\\
			&L^{\infty}([0,T];\mathbb{R}^n):=\bigg\{\mathbb{R}^n\mbox{-valued function }\phi(\cdot);\ \sup\limits_{0\leq t\leq T}|\phi(t)|dt<\infty\bigg\},\\
			&C([0,T];\mathbb{R}^n):=\bigg\{\mathbb{R}^n\mbox{-valued continuous function }\phi(\cdot);\ \sup\limits_{0\leq t\leq T}|\phi(t)|<\infty\bigg\},\\
			&L^2_{\mathcal{F}_t}(\Omega;\mathbb{R}^n):=\bigg\{\mathbb{R}^n\mbox{-valued }\mathcal{F}_t\mbox{-measurable random variable }\xi;\ \mathbb{E}|\xi|^2<\infty\bigg\},\\
			&L^2_\mathbb{F}([0,T];\mathbb{R}^n):=\bigg\{\mathbb{R}^n\mbox{-valued }\mathbb{F}\mbox{-adapted process }\phi(\cdot);\ \mathbb{E}\int_0^T|\phi(t)|^2dt<\infty  \bigg\},\\
			&L^2_\mathbb{F}\big(\Omega;C([0,T];\mathbb{R}^n)\big):=\bigg\{\mathbb{R}^n\mbox{-valued }\mathbb{F}\mbox{-adapted continuous process }\phi(\cdot);\ \mathbb{E}\Big[\sup_{0\le t\le T}|\phi(t)|^2\Big]<\infty\bigg\},\\
			&L^2\big([0,T];L^2_\mathbb{F} (0,T;\mathbb{R}^n\big):=\bigg\{\mathbb{R}^n\mbox{-valued process }\phi(\cdot,\cdot)\mbox{ such that for almost all }t\in[0,T],\\ 
			&\qquad\qquad\qquad\qquad\qquad\qquad\phi(t,\cdot)\in L^2_\mathbb{F}([0,T];\mathbb{R}^n);\ \mathbb{E}\int_0^T\int_0^T|\phi(t,s)|^2dsdt<\infty\bigg\},\\
	\end{aligned}\end{equation*}
	\begin{equation*}\begin{aligned}	
			&L_{\mathcal{F}_t}^{\infty}\big([0,T];L_{\mathbb{F}}^{\infty}([0,T];\mathbb{R}^{n\times n})\big):=\bigg\{\mathbb{R}^{n\times n}\mbox{-valued and measurable process }\phi(\cdot,\cdot)\mbox{ such that for}\\ 
			&\quad \mbox{almost all }t\in[0,T],\ \phi(t,\cdot)\mbox{ is }\mathbb{F}\mbox{-progressively}\mbox{ measurable};\ \mathbb{E}\Big[\sup_{0\le t\le T}\sup_{0\le s\le t}|\phi(t,s)|\Big]<\infty\bigg\},\\
			&L^{\infty}\big([0,T];L_{\mathbb{F}}^{\infty}([0,T];\mathbb{R}^{n\times n})\big):=\bigg\{\mathbb{R}^{n\times n}\mbox{-valued process }\phi(\cdot,\cdot)\mbox{ such that for each }t\in[0,T],\\ 
			&\quad\phi(t,\cdot)\mbox{ is }\mathbb{F}\mbox{-progressively measurable }[0,T];\sup_{0\le t\le T}\mathbb{E}\Big[\sup_{0\le s\le t}|\phi(t,s)|\Big]<\infty\bigg\},\\
			&L^{\infty}\big([0,T];L_{\mathbb{F}}^{\infty}\big(\Omega,C( [0,T];\mathbb{R}^{n\times n})\big)\big):=\bigg\{\mathbb{R}^{n\times n}\mbox{-valued process }\phi(\cdot,\cdot)\mbox{ such that for each }\\
			&\quad t\in[0,T],\phi(t,\cdot)\mbox{ is }\mathbb{F}\mbox{-progressively measurable and continuous};\ \sup_{0\le t\le T}\sup_{0\le s\le t}\mathbb{E}[|\phi(t,s)|]<\infty\bigg\}.
	\end{aligned}\end{equation*}
	
	Consider the following SDDE
	\begin{equation}\left\{\begin{aligned}\label{SDDE2}
			d\tilde{x}(t)&=\tilde{b}\Big(t,\tilde{x}(t),\tilde{x}(t-\delta),\int_{t_0}^t\tilde{\phi}(t,s)\tilde{x}(s)ds,\int_{t_0}^t\tilde{\psi}(t,s)\tilde{x}(s)dW(s)\Big)dt\\
			&\quad+\tilde{\sigma}\Big(t,\tilde{x}(t),\tilde{x}(t-\delta),\int_{t_0}^t\tilde{\phi}(t,s)\tilde{x}(s)ds,\int_{t_0}^t\tilde{\psi}(t,s)\tilde{x}(s)dW(s)\Big)dW(t),\ t\in[t_0,T],\\
			\tilde{x}(t)&=\tilde{\xi}(t),\ t\in[t_0-\delta,t_0],
		\end{aligned}\right.\end{equation}
	where $0\le t_0<T<\infty$, $\delta>0$ is given finite time delay, $\tilde{\xi}(\cdot)\in C\big([t_0-\delta,t_0];\mathbb{R}^n\big)$ is the given initial path of the state $\tilde{x}(\cdot)$, and $\tilde{\phi}(\cdot,\cdot),\tilde{\psi}(\cdot,\cdot)\in L_{\mathcal{F}_t}^{\infty}\big([t_0,T];L_{\mathbb{F}}^{\infty}(t_0,T;\mathbb{R}^{n\times n})\big)$. In the above, $\tilde{b},\tilde{\sigma}:[t_0,T]\times\mathbb{R}^n\times\mathbb{R}^n\times\mathbb{R}^n\times\mathbb{R}^n\to\mathbb{R}^n$ are given functions satisfying:
	
	(\textbf{H2.1}) there exists some $L>0$, such that for all $t\in[t_0,T]$,
	\begin{equation*}
		\begin{aligned}
			&|\tilde{b}(t,x_1,y_1,z_1,\kappa_1)-\tilde{b}(t,x_2,y_2,z_2,\kappa_2)|+|\tilde{\sigma}(t,x_1,y_1,z_1,\kappa_1)-\tilde{\sigma}(t,x_2,y_2,z_2,\kappa_2)|\\
			&\le L\big(|x_1-x_2|+|y_1-y_2|+|z_1-z_2|+|\kappa_1-\kappa_2|\big),\quad \forall x_1,x_2,y_1,y_2,z_1,z_2,\kappa_1,\kappa_2\in\mathbb{R}^n.
		\end{aligned}
	\end{equation*}
	
	(\textbf{H2.2}) $\text{sup}_{0\le t\le T}\big(|\tilde{b}(t,0,0,0,0)|+|\tilde{\sigma}(t,0,0,0,0)|\big)<+\infty$.
	
	It is established as follows that the existence and uniqueness result for SDDE (\ref{SDDE2}).
	
	\begin{mypro}\label{themSSDE}
		Assuming (\textbf{H2.1}) and (\textbf{H2.2}) are satisfied, let $\tilde{\xi}(\cdot):\Omega\to C\big([t_0-\delta,t_0]; \mathbb{R}^n\big)$ be $\mathcal{F}_{t_0}$-measurable and $\mathbb{E}\big[\sup_{t\in [t_0-\delta,t_0]}|\tilde{\xi}(t)|^2\big]<\infty$, with $\tilde{\phi}(\cdot,\cdot),\tilde{\psi}(\cdot,\cdot)\in L_{\mathcal{F}_t}^{\infty}\big(t_0,T;L_{\mathbb{F}}^{\infty}(t_0,T;\mathbb{R}^{n\times n})\big)$. Then the SDDE (\ref{SDDE2}) admits a unique solution $\tilde{x}(\cdot)\in L^2_\mathbb{F}\big(\Omega;C([t_0,T];\mathbb{R}^n)\big)$.	
	\end{mypro}
	\begin{proof}
		For any $\beta>0$, denote
		\begin{equation*}
			\mathcal{M}_\beta\big[t_0,T\big]\equiv L_\mathbb{F}^2\big(\Omega;C([t_0,T];\mathbb{R}^n)\big),
		\end{equation*}
		whose norm is defined by
		\begin{equation*}
			||x(\cdot)||_{\mathcal{M}_\beta[t_0,T]}\equiv\Big(\mathbb{E}\Big[\int_{t_0-\delta}^Te^{-\beta s}|x(s)|^2ds\Big]\Big)^{\frac{1}{2}}.
		\end{equation*}
		For any $\tilde{x}(\cdot)\in\mathcal{M}_\beta[t_0,T]$ with $\tilde{x}(t)=\tilde{\xi}(t-t_0), t\in[t_0-\delta,t_0]$, constructing the mapping 
		\begin{equation*}\begin{aligned}
				\mathcal{T}:\mathcal{M}_\beta[t_0,T]&\to\mathcal{M}_\beta[t_0,T],\\
				\tilde{x}(\cdot)&\mapsto\tilde{X}(\cdot),	
		\end{aligned}\end{equation*}
		where $\tilde{X}(\cdot)$ corresponds to the solution of the following SDDE
		\begin{equation*}\left\{\begin{aligned}
				d\tilde{X}(t)&=\tilde{b}\Big(t,\tilde{x}(t),\tilde{x}(t-\delta),\int_{t_0}^t\tilde{\phi}(t,s)\tilde{x}(s)ds,\int_{t_0}^t\tilde{\psi}(t,s)\tilde{x}(s)dW(s)\Big)dt\\
				&\quad +\tilde{\sigma}\Big(t,\tilde{x}(t),\tilde{x}(t-\delta),\int_{t_0}^t\tilde{\phi}(t,s)\tilde{x}(s)ds,\int_{t_0}^t\tilde{\psi}(t,s)\tilde{x}(s)dW(s)\Big)dW(t),\ t\in[t_0,T],\\
				\tilde{X}(t)&=\tilde{\xi}(t-t_0),\ t\in[t_0-\delta, t_0].
			\end{aligned}\right.\end{equation*}
		For any $\tilde{x}_1(\cdot),\tilde{x}_2(\cdot)\in\mathcal{M}_\beta[t_0,T]$, and $\tilde{x}_1(t),\tilde{x}_2(t)=\tilde{\xi}(t-t_0),t\in[t_0-\delta, t_0]$, let
		\begin{equation*}
			\tilde{X}_1(\cdot)=\mathcal{T}\big(\tilde{x}_1(\cdot)\big),\ \tilde{X}_2(\cdot)=\mathcal{T}\big(\tilde{x}_2(\cdot)\big),\ \breve{X}(\cdot)=\tilde{X}_1(\cdot)-\tilde{X}_2(\cdot),\ \breve{x}(\cdot)=\tilde{x}_1(\cdot)-\tilde{x}_2(\cdot).
		\end{equation*}
		Then, applying It\^o formula to $t\mapsto e^{-\beta t}|\breve{X}(t)|^2$ on $[t_0,T]$ results in
		\begin{equation*}\begin{aligned}
				&\beta\mathbb{E}\Big[\int_{t_0}^T\big<e^{-\beta t},|\breve{X}(t)|^2\big>dt\Big]\\ 
				\le&\ 2\mathbb{E}\int_{t_0}^T\Big<e^{-\beta t},\breve{X}(t)\Big|\tilde{b}\Big(t,\tilde{x}_1(t),\tilde{x}_1(t-\delta),\int_{t_0}^t\tilde{\phi}(t,s)\tilde{x}_1(s)ds,\int_{t_0}^t\tilde{\psi}(t,s)\tilde{x}_1(s)dW(s)\Big)\\
				&\qquad-\tilde{b}\Big(t,\tilde{x}_2(t),\tilde{x}_2(t-\delta),\int_{t_0}^t\tilde{\phi}(t,s)\tilde{x}_2(s)ds,\int_{t_0}^t\tilde{\psi}(t,s)\tilde{x}_2(s)dW(s)\Big)   \Big|\Big>dt\\
				&+\mathbb{E}\int_{t_0}^Te^{-\beta t}\Big|\tilde{\sigma}\Big(t,\tilde{x}_1(t),\tilde{x}_1(t-\delta),\int_{t_0}^t\tilde{\phi}(t,s)\tilde{x}_1(s)ds,\int_{t_0}^t\tilde{\psi}(t,s)\tilde{x}_1(s)dW(s)\Big)\\
				&\qquad-\tilde{\sigma}\Big(t,\tilde{x}_2(t),\tilde{x}_2(t-\delta),\int_{t_0}^t\tilde{\phi}(t,s)\tilde{x}_2(s)ds,\int_{t_0}^t\tilde{\psi}(t,s)\tilde{x}_2(s)dW(s)\Big)\Big|^2dt.
		\end{aligned}\end{equation*}
		Notice that
		\begin{equation*}\begin{aligned}
				&\int_{t_0}^Te^{-\beta t}|\breve{x}(t-\delta)|^2dt\le \hat{L}\int_{t_0-\delta}^Te^{-\beta t}|\breve{x}(t)|^2dt,\\
				&\mathbb{E}\int_{t_0}^Te^{-\beta t}\Big|\int_{t_0}^t\tilde{\phi}(t,s)\breve{x}(s)ds\Big|^2dt\le\mathbb{E}\int_{t_0}^Te^{-\beta t}\int_{t_0}^t\big|\tilde{\phi}(t,s)\breve{x}(s)\big|^2dsdt\\
				\le&\ \mathbb{E}\int_{t_0}^T\big|\breve{x}(t)\big|^2\int_t^Te^{-\beta s}\big|\tilde{\phi}(s,t)\big|^2dsdt\le \bar{L}^2T\mathbb{E}\int_{t_0}^Te^{-\beta t}\big|\breve{x}(t)\big|^2dt,\\
				&\ \mathbb{E}\int_{t_0}^Te^{-\beta t}\Big|\int_{t_0}^t\tilde{\psi}(t,s)\breve{x}(s)dW(s)\Big|^2ds= \mathbb{E}\int_{t_0}^Te^{-\beta t}\int_{t_0}^t\big|\tilde{\psi}(t,s)\breve{x}(s)\big|^2dsdt\\
				\le&\ \mathbb{E}\int_{t_0}^T\big|\breve{x}(t)\big|^2\int_t^Te^{-\beta s}\big|\tilde{\psi}(s,t)\big|^2dsdt\le\bar{L}^2T\mathbb{E}\int_{t_0}^Te^{-\beta t}\big|\breve{x}(t)\big|^2dt,
		\end{aligned}\end{equation*}
		where $\hat{L}, \bar{L}$ are finite constants. Then, by (\textbf{H2.1}) and Cauchy-Schwartz inequality, one has
		\begin{equation*}\begin{aligned}
				&(\beta-1)\mathbb{E}\int_{t_0}^T\big<e^{-\beta t},|\breve{X}(t)|^2\big>dt\\
				\le &\ 8L^2\mathbb{E}\int_{t_0}^T\Big<e^{-\beta t},|\breve{x}(t)|^2+|\breve{x}(t-\delta)|^2+\Big|\int_{t_0}^t\tilde{\phi}(t,s)\breve{x}(s)ds\Big|^2+\Big|\int_{t_0}^t\tilde{\psi}(t,s)\breve{x}(s)dW(s)\Big|^2\Big>dt\\
				\le &\ 8L^2(1+\hat{L}+2\bar{L}^2T)\mathbb{E}\int_{t_0-\delta}^T\big<e^{-\beta t},|\breve{x}(t)|^2\big>dt.
		\end{aligned}\end{equation*}
		Choose $\beta=16L^2(1+\hat{L}+2\bar{L}^2T)+1$, then
		\begin{equation*}
			\mathbb{E}\int_{t_0-\delta}^T\big<e^{-\beta t},|\breve{X}(t)|^2\big>dt\le \frac{1}{2}\mathbb{E}\int_{t_0-\delta}^T\big<e^{-\beta t},|\breve{x}(t)|^2\big>dt,
		\end{equation*}
		which implies that $\mathcal{T}$ is a strict contraction mapping. Then the SDDE (\ref{SDDE2}) admits a unique solution in
		$L^2_\mathbb{F}\big(\Omega;C([t_0,T];\mathbb{R}^n)\big)$ by the fixed point theorem.
	\end{proof}	
	
	Let $\mathbb{R}^+$ denote the set of nonnegative real numbers. We consider the following ABSDEs
	\begin{equation}\left\{\begin{aligned}\label{ABSDE1}
			-dY(t)&=\tilde{g}\big(t,Y(t),Z(t),Y(t+\zeta_1(t)),Z(t+\zeta_2(t))\big)dt-Z(t)dW(t),\ t\in[0,T],\\
			Y(t)&=\alpha(t),\ Z(t)=\beta(t),\ t\in[T,T+K],
		\end{aligned}\right.\end{equation}
	where the terminal conditions $\alpha(\cdot)\in L^2_\mathbb{F}\big(\Omega;C([T,T+K];\mathbb{R}^n)\big)$ and $\beta(\cdot)\in L^2_\mathbb{F}(\Omega;T,T+K;\mathbb{R}^{n})$ are given. Let $\zeta_1(\cdot)$ and $\zeta_2(\cdot)$ defined on $[0,T]$ be given $\mathbb{R}^+$-valued functions such that
	
	(\textbf{H2.3}) (i) there exists a constant $K\ge0$ such that for all $s\in [0,T]$, $s+\zeta_1(s)\le T+K$, $s+\zeta_2(s)\le T+K$.
	
	\hspace{1.2cm}(ii) there exists a constant $M\ge0$ such that for all $t\in [0,T]$ and all nonnegative with integrable function $\tilde{l}(\cdot)$, 
	\begin{equation*}
		\int_t^T\tilde{l}\big(s+\zeta_1(s)\big)ds\le M\int_t^{T+K}\tilde{l}(s)ds,\ \int_t^T\tilde{l}\big(s+\zeta_2(s)\big)ds\le M\int_t^{T+K}\tilde{l}(s)ds.
	\end{equation*}
	
	The conditions imposed on the generator $\tilde{g}$ of ABSDEs (\ref{ABSDE1}) are as follows:
	
	(\textbf{H2.4}) $\tilde{g}(s,\omega,y,z,\alpha,\beta):\Omega\times\mathbb{R}^n\times\mathbb{R}^{n}\times L^2_{\mathcal{F}_{r_1}}(\Omega;\mathbb{R}^n)\times L^2_{\mathcal{F}_{r_2}}(\Omega;\mathbb{R}^{n})\to L^2_{\mathcal{F}_s}(\Omega;\mathbb{R}^n)$ for all $s\in[0,T]$, where $r_1,r_2\in[s,T+K]$, and $\mathbb{E}\big[\int_0^T|\tilde{g}(s,0,0,0,0)|^2ds\big]<+\infty$.
	
	(\textbf{H2.5}) There exists a constant $C>0$ such that for all $s\in[0,T]$, $y_1,y_2,z_1, z_2\in\mathbb{R}^{n}$, $\alpha_1(\cdot), \alpha_2(\cdot)\in L^2_\mathbb{F}\big(\Omega;C\big([s,T+K];\mathbb{R}^n\big)\big)$, $\beta_1(\cdot), \beta_2(\cdot)\in L^2_\mathbb {F}(\Omega;s,T+K;\mathbb{R}^{n})$, $r_1,r_2\in[s,T+K]$, the following holds
	\begin{equation*}\begin{aligned}
			&\big|g\big(s,y_1,z_1,\alpha_1(r_1),\beta_1(r_2)\big)-g\big(s,y_2,z_2,\alpha_2(r_1),\beta_2(r_2)\big)\big|\\
			\le~& C\big(|y_1-y_2|+|z_1-z_2|+\mathbb{E}^{\mathcal{F}_s}\big[|\alpha_1(r_1)-\alpha_2(r_1)|+|\beta_1(r_2)-\beta_2(r_2)|\big]\big).
	\end{aligned}\end{equation*}
	
	\begin{mypro}\label{ABSDE exist}
		(see \cite{PY09})
		For any given $\big(\alpha(\cdot),~\beta(\cdot)\big)\in L^2_\mathbb{F}\big(\Omega;C\big([T,T+K];\mathbb{R}^n\big)\big)\times  L^2_\mathbb{F}(\Omega;T,T+K;\mathbb{R}^{n})$, suppose that (\textbf{H2.3})-(\textbf{H2.5}) are satisfied. Then,  (\ref{ABSDE1}) admits a unique $\mathcal{F}_t$-adapted solution pair $\big(Y(\cdot), Z(\cdot)\big)\in L^2_\mathbb{F}\big(\Omega;C\big([0,T+K];\mathbb{R}^n\big)\big)\times L^2_\mathbb{F}(\Omega;0,T+K;\mathbb{R}^{n})$.	
	\end{mypro}
	
	Consider the following BSVIE
	\begin{equation}\label{BSVIE}
		\tilde{Y}(t)=\tilde{\psi}(t)+\int_t^T\check{g}\big(t,s,\tilde{Y}(s),\tilde{Z}(t,s),\tilde{Z}(s,t)\big)ds-\int_t^T\tilde{Z}(t,s)dW(s),\ t\in[0,T],
	\end{equation}
	where $\check{g}$ is the given function satisfying
	
	(\textbf{H2.6}) $\check{g}$ is $\mathcal{B}\big([0,T]^2\times\mathbb{R}^n\times\mathbb{R}^{n}\times\mathbb{R}^{n} \big)\otimes\mathcal{F}_T$-measurable such that $s\mapsto \check{g}\big(t,s,y,z,\eta\big)$ is progressively measurable and for all $ 0\le t\le s\le T,\ y_1, y_2, z_1,z_2,\eta_1,\eta_2\in\mathbb{R}^{n}$, it holds true
	\begin{equation*}
		\big|\check{g}\big(t,s,y_1,z_1,\eta_1\big)-\check{g}\big(t,s,y_2,z_2,\eta_2\big)\big|\le L(t,s)\big(|y_1-y_2|+|z_1-z_2|+|\eta_1-\eta_2| \big),~a.s.,   	
	\end{equation*}
	where $L$ is a deterministic function such that for some $\varepsilon >0$,
	\begin{equation*}
		\sup_{t\in[0,T]}\int_t^TL(t,s)^{2+\varepsilon}ds<\infty.
	\end{equation*}
	
	In addition, for all $(t,y,z,\eta)\in [0,T]\times\mathbb{R}^n\times\mathbb{R}^{n}\times\mathbb{R}^{n}$, one has
	\begin{equation*}
		\mathbb{E}\int_0^T\bigg(\int_t^T|\check{g}(t,s,0,0,0)|ds \bigg)^2dt<\infty.
	\end{equation*}

	\begin{mypro}\label{BSVIE exist}
		(see \cite{YJM08}) Suppose that (\textbf{H2.6}) holds. Then, for any $\mathcal{B}\big([0,T]\big)\otimes\mathcal{F}_T$-measurable process $\tilde{\psi}(\cdot)$ such that $\mathbb{E}\int_0^T|\tilde{\psi}(t)|^2dt<\infty$, the BSVIE (\ref{BSVIE}) admits a unique adapted solution pair $\big(\tilde{Y}(\cdot), \tilde{Z}(\cdot,\cdot)\big)\in L^2_\mathbb{F} (0,T;\mathbb{R}^n)\times L^2\big(0,T;L^2_\mathbb{F}(0,T;\mathbb{R}^{n})\big)$ satisfying 
		\begin{equation*}
			\tilde{Y}(t)=\mathbb{E}^{\mathcal{F}_s}[\tilde{Y}(t)]+\int_s^t\tilde{Z}(t,r)dW(r),\ a.e.~t\in [s,T].
		\end{equation*}
	\end{mypro}
	
	In what follows, we recall several concepts related to the generalized Malliavin calculus associated with Brownian motion, which will be used in the subsequent analysis.
	
	Since we prefer not to conduct a lengthy survey on the literature of Malliavin calculus, let us just list some books \cite{DN06, SS05, NOP09}, where good surveys and tutorials along with extensive references can be found. Indeed, the framework developed by Aase et al. \cite{AOPU00} has been extended from the classical Malliavin differentiability space $\mathbb{D}_{1,2}$ to $L^2_{\mathcal{F}_T}(\Omega;\mathbb{R}^n)$,
	where $\mathbb{D}_{1,2}$ denotes the space of Malliavin differentiable ${\mathcal{F}_T}$-measurable random variables. The extension is such that for all $F\in L^2_{\mathcal{F}_T}(\Omega;\mathbb{R}^n)$, the following assertions are satisfied
	
	(i) $D_tF\in(\mathcal{S})^\ast$, where $(\mathcal{S})^\ast\supseteq L^2_{\mathcal{F}_T}(\Omega;\mathbb{R}^n)$ denotes the Hida space of stochastic distributions.
	
	(ii) The map $(t,\omega)\mapsto\mathbb{E}^{\mathcal{F}_t}[D_tF]\in L^2_{\mathbb{F}}(0,T;\mathbb{R}^n)$.
	
	(iii) In addition, the following  \emph{generalized Clark-Ocone Theorem} holds
	\begin{equation*}
		F=\mathbb{E}[F]+\int_0^T\mathbb{E}^{\mathcal{F}_t}[D_tF]dW(t).
	\end{equation*}
	
	Then, combining with It\^o's isometry and the Clark-Ocone Theorem result in 
	\begin{equation}\label{Clarke-Ocone theorem}
		\mathbb{E}\Big[\int_0^T\mathbb{E}\big[D_tF|\mathcal{F}_t\big]^2dt\Big]=\mathbb{E}\Big[\Big(\int_0^T\mathbb{E}^{\mathcal{F}_t}[D_tF]dW(t)\Big)^2\Big]=\mathbb{E}\big[F^2-\mathbb{E}[F]^2\big].
	\end{equation}
	
	\begin{mypro}
		(see \cite{NOP09}) (Generalized duality formula) Suppose that $F\in L^2_{\mathcal{F}_T}(\Omega;\mathbb{R}^n)$ with $\tilde{\varphi}(\cdot)\in L^2_{\mathbb{F}}(0,T;\mathbb{R}^n)$, then it holds 
		\begin{equation}\label{duality formula}
			\mathbb{E}\Big[F\int_0^T\tilde{\varphi}(t)dW(t)\Big]=\mathbb{E}\Big[\int_0^T \mathbb{E}^{\mathcal{F}_t}[D_tF\big]\tilde{\varphi}(t)dt\Big].
		\end{equation}
	\end{mypro}
	
	\begin{mypro}\label{pro2.5}
		see(\cite{KPQ97}) For any given terminal condition $\gamma\in L^2_{\mathcal{F}_T}(\Omega,\mathbb{R}^n)$, assume that the adapted process pair  $\big(\breve{Y}(\cdot),\breve{Z}(\cdot)\big)\in L^2\big(\Omega;C([0,T];\mathbb{R}^n)\big)\times L^2\big(\Omega;0,T;\mathbb{R}^{n}\big)$ is the solution of the following BSDE
		\begin{equation}\label{EPQ1997 BSDE}
			\left\{\begin{aligned}
				-d\breve{Y}(t)&=\breve{g}\big(t,\breve{Y}(t),\breve{Z}(t)\big)dt+\breve{Z}(t)dW(t),\ t\in[0,T],\\
				\breve{Y}(T)&=\gamma.
			\end{aligned}
			\right.\end{equation}
		Then, it holds 
		\begin{equation*}
			\breve{Z}(t)=D_t\breve{Y}(t),
		\end{equation*}
		where $D$ denotes Malliavin derivative.
	\end{mypro}

\normalsize

\section{Transform SDDE into SVIE}

In this section, we derive the variational equation to be studied, then the \textbf{Problem (P)} is converted into an optimal control problem without delay, which is driven by an SVIE. Similar transformation also presented in \cite{HY23, MSWZ25}.

The following assumptions are imposed in this work.

(\textbf{H3.1})
(i) The map $(x,y,z,\kappa,u,\mu,\nu,\lambda)\mapsto b=b(t,x,y,z,\kappa,u,\mu,\nu,\lambda),\  \sigma=\sigma(t,x,y,z,\kappa,u,\\\mu,\nu,\lambda),\  l=l(t,x,y,z,\kappa,u,\mu,\nu,\lambda),\  h=h(t,x,y,z,\kappa)$ are continuous differentiable in $(x,y,z,\kappa,\\u,\mu,\nu,\lambda)$. They with their derivatives $f_\varrho$ are continuous and uniformly bounded in $(x,y,z,\kappa,u,\mu,\\\nu,\lambda)$, where $\varrho=x,y,z,\kappa,u,\mu,\nu,\lambda$, $f=b,\sigma,l,h$.

(ii) $|b(t,0,0,0,0,u,\mu,\nu,\lambda)|+|\sigma(t,0,0,0,0,u,\mu,\nu,\lambda)|+|l(t,0,0,0,0,u,\mu,\nu,\lambda)|+|h(0,0,0,0)|+|l_{\varrho}(t,0,0,0,0,u,\mu,\nu,\lambda)|+|h_{\varrho}(0,0,0,0)|\le C$ for some $C$, where $\varrho=x,y,z,\kappa,u,\mu,\nu,\lambda$.

(iii) The initial trajectory of the state $\xi(\cdot)$ is a deterministic continuous function, and the initial trajectory of the control $\varsigma(\cdot)$ is a deterministic square integrable function. The autoregressive kernel $\phi_1(\cdot,\cdot)$ and the moving average kernel $\psi_1(\cdot,\cdot)$ belong to $L^{\infty}\big([t_0,T];L_{\mathbb{F}}^{\infty}([t_0,T];\mathbb{R}^{n\times n})\big),$ and $\phi_2(\cdot,\cdot),\psi_2(\cdot,\cdot)\in L^{\infty}\big([t_0,T];L_{\mathbb{F}}^{\infty}\big(\Omega,C( [t_0,T];\mathbb{R}^{n\times n})\big)\big)$.

(iv) There exists a constant $C$ such that
$\big|f(t,x,y,z,\kappa,u,\mu,\nu,\lambda)\big|\le C\big(|x|+|y|+|z|+|\kappa|+|u|+|\mu|+|\nu|+|\lambda|\big)$ for any $x,y,z,\kappa, u,\mu,\nu,\lambda$, where $f=b,\sigma,l,h$.

(v) $b=b(t,x,y,z,\kappa,u,\mu,\nu,\lambda),\  \sigma=\sigma(t,x,y,z,\kappa,u,\mu,\nu,\lambda),\  l=l(t,x,y,z,\kappa,u,\mu,\nu,\lambda)$ is $\mathbb{F}$-progressively measurable for any $x,y,z,\kappa,u,\mu,\nu,\lambda$. $h=h(x,y,z,\kappa)$ is $\mathcal{F}_T$-measurable for any $x,y,z,\kappa$.

Under (\textbf{H3.1}), the SDDE (\ref{state equation}) admits unique solution by Proposition \ref{themSSDE}. In addition, the cost function (\ref{cost1}) is well-define and \textbf{Problem (P)} is meaningful.

The convexity of the control domain $U$ allows us to employ the convex variation technique in the study of \textbf{Problem (P)}. Suppose that $u^{\ast}(\cdot)$ is an optimal control with $x^{\ast}(\cdot)$ corresponding to the optimal trajectory. Let $v(\cdot)$ be such that $u^{\ast}(\cdot)+v(\cdot)\in\mathcal{U}_{ad}$, then the variational control $u^{\rho}(\cdot)$ defined below also belongs to $\mathcal{U}_{ad}$:
\begin{equation}\label{u rho}
	u^{\rho}(\cdot)=u^{\ast}(\cdot)+\rho v(\cdot),\ 0\le\rho\le 1,
\end{equation}
and the corresponding trajectory is denoted as $x^{\rho}(\cdot)$. 

For simplification, we use the following notations in this paper:
\begin{equation*}\begin{aligned}
		f^\rho(t):&=f\big(t,x^\rho(t),y^\rho(t),z^\rho(t),\kappa^\rho(t),u^\rho(t),\mu^\rho(t),\nu^\rho(t),\kappa^\rho(t)\big),\\
		f(t):&=f\big(t,x^\ast(t),y^\ast(t),z^\ast(t),\kappa^\ast(t),u^\ast(t),\mu^\ast(t),\nu^\ast(t),\kappa^\ast(t)\big),\\	
		f_{\varrho}(t):&=f_{\varrho}\big(t,x^\ast(t),y^\ast(t),z^\ast(t),\kappa^\ast(t),u^\ast(t),\mu^\ast(t),\nu^\ast(t),\kappa^\ast(t)\big),	
\end{aligned}\end{equation*}
where $f=b,\sigma,l,h$ and $\varrho=x,y,z,\kappa,u,\mu,\nu,\lambda$.

Inspired by \cite{CW10}, we introduce the variational equation:
\begin{equation}\left\{\begin{aligned}\label{variational equation}
		d\hat{x}(t)&=\Big[b_x(t)\hat{x}(t)+b_y(t)\hat{y}(t)+b_z(t)\hat{z}(t)+b_{\kappa}(t)\hat{\kappa}(t)+b_u(t)v(t)+b_{\mu}(t)v(t-\delta)\\
		&\qquad+b_{\nu}(t)\int_{t_0}^t\phi_2(t,s)v(s)ds+b_{\lambda}(t)\int_{t_0}^t\psi_2(t,s)v(s)dW(s)\Big]dt\\
		&\quad+\Big[\sigma_x(t)\hat{x}(t)+\sigma_y(t)\hat{y}(t)+\sigma_z(t)\hat{z}(t)+\sigma_{\kappa}(t)\hat{\kappa}(t)+\sigma_u(t)v(t)+\sigma_{\mu}(t)v(t-\delta)\\
		&\qquad+\sigma_{\nu}(t)\int_{t_0}^t\phi_2(t,s)v(s)ds+\sigma_{\lambda}(t)\int_{t_0}^t\psi_2(t,s)v(s)dW(s)\Big]dW(t),\ t\in[t_0,T],\\
		\hat{x}(t)&=0,\ t\in[t_0-\delta,t_0],
	\end{aligned}\right.\end{equation}
which admits a unique solution $\hat{x}(\cdot)$ by the assumption \textbf{(H3.1)} together with Proposition \ref{themSSDE}. 

We notice that, for the variational equation (\ref{variational equation}), when $\phi_2(\cdot,s),\ \psi_2(\cdot,s)\equiv0,\ s< t_0$ holds, 
then it can be viewed as an extension of the SDDE as follows:
\begin{equation}\left\{\begin{aligned}\label{exvariational equation}
		d\hat{x}(t)&=\Big[b_x(t)\hat{x}(t)+b_y(t)\hat{y}(t)+b_z(t)\check{z}(t)+b_{\kappa}(t)\check{\kappa}(t)+b_u(t)v(t)+b_{\mu}(t)v(t-\delta)\\
		&\qquad+b_{\nu}(t)\int_{t-\delta}^t\phi_2(t,s)v(s)ds+b_{\lambda}(t)\int_{t-\delta}^t\psi_2(t,s)v(s)dW(s)\Big]dt\\
		&\quad+\Big[\sigma_x(t)\hat{x}(t)+\sigma_y(t)\hat{y}(t)+\sigma_z(t)\check{z}(t)+\sigma_{\kappa}(t)\check{\kappa}(t)+\sigma_u(t)v(t)+\sigma_{\mu}(t)v(t-\delta)\\
		&\qquad+\sigma_{\nu}(t)\int_{t-\delta}^t\phi_2(t,s)v(s)ds+\sigma_{\lambda}(t)\int_{t-\delta}^t\psi_2(t,s)v(s)dW(s)\Big]dW(t),\ t\in[t_0,T],\\
		\hat{x}(t)&=0,\ t\in[t_0-\delta,t_0],\\
		\check{z}(t)&=\int_{t-\delta}^t\phi_1(t,s)\hat{x}(s)ds,\ \check{\kappa}(t)=\int_{t-\delta}^t\psi_1(t,s)\hat{x}(s)dW(s),\ t\in[t_0,T],
	\end{aligned}\right.\end{equation}
where $\hat{y}(\cdot)=\hat{x}(\cdot-\delta)$. In fact,
\begin{equation*}\begin{aligned}
		\check{z}(t)=\int_{t_0}^t\tilde{\phi}_1(t,s)\hat{x}(s)ds,&\quad \check{\kappa}(t)=\int_{t_0}^t\tilde{\psi}_1(t,s)\hat{x}(s)dW(s),\\
		\int_{t-\delta}^t\phi_2(t,s)v(s)ds=\int_{t_0}^t\tilde{\phi}_2(t,s)v(s)ds,&\quad \int_{t-\delta}^t\psi_2(t,s)v(s)dW(s)=\int_{t_0}^t\tilde{\psi}_2(t,s)v(s)dW(s),
\end{aligned}\end{equation*}
where
\begin{equation*}\begin{aligned}
		\tilde{\phi}_1(t,s)&\equiv\big[\mathbb{I}_{[t_0,t_0+\delta)}(t)+\mathbb{I}_{[t+\delta,T]}(t)\mathbb{I}_{[t-\delta,t)}(s)\big]\phi_1(t,s),\\
		\tilde{\psi}_1(t,s)&\equiv\big[\mathbb{I}_{[t_0,t_0+\delta)}(t)+\mathbb{I}_{[t+\delta,T]}(t)\mathbb{I}_{[t-\delta,t)}(s)  \big]\psi_1(t,s),\\
		\tilde{\phi}_2(t,s)&\equiv\big[\mathbb{I}_{[t_0,t_0+\delta)}(t)+\mathbb{I}_{[t+\delta,T]}(t)\mathbb{I}_{[t-\delta,t)}(s)  \big]\phi_2(t,s),\\	
		\tilde{\psi}_2(t,s)&\equiv\big[\mathbb{I}_{[t_0,t_0+\delta)}(t)+\mathbb{I}_{[t+\delta,T]}(t)\mathbb{I}_{[t-\delta,t)}(s)  \big]\psi_2(t,s).
\end{aligned}\end{equation*}
In view of the above analysis, (\ref{state equation}) is termed as a \emph{(generalized) controlled SDDE with extended mixed delays}. It is worth noting that the case $t_0+\delta\le T$ is included in our formulation.

The following estimate is developed in the spirit of Lemma 4.3 in \cite{DMOR16} and Lemma 3.1 in \cite{MSWZ25}.

\begin{mylem}\label{lem3.1}
	Let \textbf{(H3.1)} hold. Then the following result is obtained
	\begin{equation}
		\lim_{\rho\to 0}\mathbb{E}\Big[\sup_{t\in[t_0,T]}|\check{X}(t)|^2\Big]=0,
	\end{equation}
	where $\check{X}(t):=\frac{x^{\rho}(t)-x^{\ast}(t)}{\rho}-\hat{x}(t),\ t\in[t_0,T].$
\end{mylem}
\begin{proof}
	By the definition of $\check{X}(t)$, we can derive 
	\begin{equation}\left\{\begin{aligned}
			d\check{X}(t)=&\Big[A^{\rho}(t)\check{X}(t)+B^{\rho}(t)\check{X}(t-\delta)+C^{\rho}(t)\int_{t_0}^t\phi_1(t,s)\check{X}(s)ds\\
			&\ +D^{\rho}(t)\int_{t_0}^t\psi_1(t,s)\check{X}(s)dW(s)+R(t)\Big]dt\\
			&+\Big[I^{\rho}(t)\check{X}(t)+K^{\rho}(t)\check{X}(t-\delta)+L^{\rho}(t)\int_{t_0}^t\phi_1(t,s)\check{X}(s)ds\\
			&\ +M^{\rho}(t)\int_{t_0}^t\psi_1(t,s)\check{X}(s)dW(s)+S(t)\Big]dW(t),\ t\in [t_0,T],\\
			\check{X}(t)=&\ 0,\ t\in[t_0-\delta, t_0],
		\end{aligned}\right.\end{equation}
	where
	\begin{equation*}\begin{aligned}
			A^{\rho}(t):=&\int_0^1 b_x\big(t,x^{\ast}(t)+\theta[x^{\rho}(t)-x^{\ast}(t)],y^{\rho}(t),z^{\rho}(t),\kappa^{\rho}(t),u^{\rho}(t),\mu^{\rho}(t),\nu^{\rho}(t),\lambda^{\rho}(t)\big)d\theta,\\
			B^{\rho}(t):=&\int_0^1 b_y\big(t,x^{\ast}(t),y^{\ast}(t)+\theta[y^{\rho}(t)-y^{\ast}(t)],z^{\rho}(t),\kappa^{\rho}(t),u^{\rho}(t),\mu^{\rho}(t),\nu^{\rho}(t),\lambda^{\rho}(t)\big)d\theta,\\
			C^{\rho}(t):=&\int_0^1 b_z\big(t,x^{\ast}(t),y^{\ast}(t),z^{\ast}(t)+\theta[z^{\rho}(t)-z^{\ast}(t)],\kappa^{\rho}(t),u^{\rho}(t),\mu^{\rho}(t),\nu^{\rho}(t),\lambda^{\rho}(t)\big)d\theta,\\
	\end{aligned}\end{equation*}
	\begin{equation*}\begin{aligned}
			D^{\rho}(t):=&\int_0^1 b_{\kappa}\big(t,x^{\ast}(t),y^{\ast}(t),z^{\ast}(t),\kappa^{\ast}(t)+\theta[\kappa^{\rho}(t)-\kappa^{\ast}(t)],u^{\rho}(t),\mu^{\rho}(t),\nu^{\rho}(t),\lambda^{\rho}(t)\big)d\theta,\\
			E^{\rho}(t):=&\int_0^1 b_{u}\big(t,x^{\ast}(t),y^{\ast}(t),z^{\ast}(t),\kappa^{\ast}(t),u^{\ast}(t)+\rho\theta v(t),\mu^{\rho}(t),\nu^{\rho}(t),\lambda^{\rho}(t)\big)d\theta,\\ 
			F^{\rho}(t):=&\int_0^1 b_{\mu}\big(t,x^{\ast}(t),y^{\ast}(t),z^{\ast}(t),\kappa^{\ast}(t),u^{\ast}(t),\mu^{\ast}(t)+\rho\theta v(t-\delta),\nu^{\rho}(t),\lambda^{\rho}(t)\big)d\theta,\\
			G^{\rho}(t):=&\int_0^1 b_{\nu}\Big(t,x^{\ast}(t),y^{\ast}(t),z^{\ast}(t),\kappa^{\ast}(t),u^{\ast}(t),\mu^{\ast}(t),\nu^{\ast}(t)+\rho\theta\int_{t_0}^t\phi_2(t,s)v(s)ds,\lambda^{\rho}(t)\Big)d\theta,\\
			H^{\rho}(t):=&\int_0^1 b_{\lambda}\Big(t,x^{\ast}(t),y^{\ast}(t),z^{\ast}(t),\kappa^{\ast}(t),u^{\ast}(t),\mu^{\ast}(t),\nu^{\ast}(t),\lambda^{\ast}(t)+\rho\theta\int_{t_0}^t\psi_2(t,s)v(s)dW(s)\Big)d\theta,\\
			R^{\rho}(t):=&\ A^{\rho}(t)\hat{x}(t)+B^{\rho}(t)\hat{y}(t)+C^{\rho}(t)\hat{z}(t)+D^{\rho}(t)\hat{\kappa}(t)+E^{\rho}(t)v(t)+F^{\rho}(t)v(t-\delta)\\
			& +G^{\rho}(t)\int_{t_0}^{t}\phi_2(t,s)v(s)ds+H^{\rho}(t)\int_{t_0}^{t}\psi_2(t,s)v(s)dW(s)-b_x(t)\hat{x}(t)-b_y(t)\hat{y}(t)\\
			& -b_z(t)\hat{z}(t)-b_{\kappa}(t)\hat{\kappa}(t)-b_u(t)v(t)-b_{\mu}(t)v(t-\delta)-b_{\nu}(t)\int_{t_0}^{t}\phi_2(t,s)v(s)ds\\
			& -b_{\lambda}(t)\int_{t_0}^{t}\psi_2(t,s)v(s)dW)(s),\\
			S^{\rho}(t):=&\ I^{\rho}(t)\hat{x}(t)+K^{\rho}(t)\hat{y}(t)+L^{\rho}(t)\hat{z}(t)+M^{\rho}(t)\hat{\kappa}(t)+N^{\rho}(t)v(t)+O^{\rho}(t)v(t-\delta)\\
			&+P^{\rho}(t)\int_{t_0}^{t}\phi_2(t,s)v(s)ds+Q^{\rho}(t)\int_{t_0}^{t}\psi_2(t,s)v(s)dW(s)-\sigma_x(t)\hat{x}(t)-\sigma_y(t)\hat{y}(t)\\
			& -\sigma_z(t)\hat{z}(t)-\sigma_{\kappa}(t)\hat{\kappa}(t)-\sigma_u(t)v(t)-\sigma_{\mu}(t)v(t-\delta)-\sigma_{\nu}(t)\int_{t_0}^{t}\phi_2(t,s)v(s)ds\\
			& -\sigma_{\lambda}(t)\int_{t_0}^{t}\psi_2(t,s)v(s)dW(s),
	\end{aligned}\end{equation*}
	and $I^{\rho}(t)$, $K^{\rho}(t)$, $L^{\rho}(t)$, $M^{\rho}(t)$, $N^{\rho}(t)$, $O^{\rho}(t)$, $P^{\rho}(t)$, $Q^{\rho}(t)$ are denoted similar to $A^{\rho}(t)$, $B^{\rho}(t)$, $C^{\rho}(t)$, $D^{\rho}(t)$, $E^{\rho}(t)$, $F^{\rho}(t)$, $G^{\rho}(t)$, $H^{\rho}(t)$ respectively, and just need to replace $b$ by $\sigma$.
	
	Applying It\^o's formula to $t\mapsto|\check{X}(t)|^2$, one has
	\begin{equation}\begin{aligned}\label{esx2}
			|\check{X}(t)|^2\le&\ \int_{t_0}^t\Big[2\Big<\check{X}(s),A^\rho(s)\check{X}(s)+B^\rho(s)\check{X}(s-\delta)+C^\rho(s)\int_{t_0}^s\phi_1(s,r)\check{X}(r)dr\\
			&+D^\rho(s)\int_{t_0}^s\psi_1(s,r)\check{X}(r)dW(r)+R(s)\Big>+\Big<I^\rho(s)\check{X}(s)+K^\rho(s)\check{X}(s-\delta)\\
			&+L^\rho(s)\int_{t_0}^s\phi_1(s,r)\check{X}(r)dr+M^\rho(s)\int_{t_0}^s\psi_1(s,r)\check{X}(r)dW(r)+S(s),I^\rho(s)\check{X}(s)\\
			&+K^\rho(s)\check{X}(s-\delta)+L^\rho(s)\int_{t_0}^s\phi_1(s,r)\check{X}(r)dr+M^\rho(s)\int_{t_0}^s\psi_1(s,r)\check{X}(r)dW(r)\\
			&+S(s)\Big>\Big]ds+\int_{t_0}^t\Big[2\Big<\check{X}(s),I^\rho(s)\check{X}(s)+K^\rho(s)\check{X}(s-\delta)\\
			&+L^\rho(s)\int_{t_0}^s\phi_1(s,r)\check{X}(r)dr+M^\rho(s)\int_{t_0}^s\psi_1(s,r)\check{X}(r)dW(r)+S(s)\Big>\Big]dW(s).
	\end{aligned}\end{equation}
	By B-D-G's inequality, we gain
	\begin{equation}\hspace{-27mm}\begin{aligned}\label{esphix2}
			&\mathbb{E}\int_{t_0}^T\Big|\int_{t_0}^s\phi_1(s,r)\check{X}(r)dr\Big|^2ds\le\mathbb{E}\int_{t_0}^T\int_{t_0}^T\big|\phi_1(s,r)\check{X}(r)\big|^2drds\\
			\le&\ \mathbb{E}\int_{t_0}^T\sup_{t\in[t_0,T]}\big|\check{X}(t)\big|^2\int_{t_0}^T\big|\phi_1(s,t)\big|^2dsdt\le C(\bar{L},T)\int_{t_0}^T\mathbb{E}\Big[\sup_{t\in[t_0,T]}\big|\check{X}(t)\big|^2\Big]dt.
	\end{aligned}\end{equation}
	\begin{equation}\hspace{-7mm}\begin{aligned}\label{espsix2}
			&\mathbb{E}\int_{t_0}^T\Big|\int_{t_0}^s\psi_1(s,r)\check{X}(r)dW(r)\Big|^2ds\le C(2)\mathbb{E}\int_{t_0}^T\int_{t_0}^T\big|\psi_1(s,r)\check{X}(r)\big|^2drds\\
			\le&\ C(2)\mathbb{E}\int_{t_0}^T\sup_{t\in[t_0,T]}\big|\check{X}(t)\big|^2\int_{t_0}^T\big|\psi_1(s,t)\big|^2dsdt\le C(2,\bar{L},T)\int_{t_0}^T\mathbb{E}\Big[\sup_{t\in[t_0,T]}\big|\check{X}(t)\big|^2\Big]dt.
	\end{aligned}\end{equation}
	\begin{equation}\begin{aligned}\label{esdiffx2}
			&\mathbb{E}\bigg[\sup_{t\in[t_0,T]}\Big|\int_{t_0}^t2\Big<\check{X}(s),I^\rho(s)\check{X}(s)+K^\rho(s)\check{X}(s-\delta)+L^\rho(s)\int_{t_0}^s\phi_1(s,r)\check{X}(r)dr\\
			&\qquad\qquad+M^\rho(s)\int_{t_0}^s\psi_1(s,r)\check{X}(r)dW(r)+S(s)\Big>dW(s)\Big|\bigg]\\
			\le&\ \mathbb{E}\bigg[\int_{t_0}^T\Big<4|\check{X}(s)|^2,\Big| I^\rho(s)\check{X}(s)+K^\rho(s)\check{X}(s-\delta)+L^\rho(s)\int_{t_0}^s\phi_1(s,r)\check{X}(r)dr\\
			&\qquad\qquad+M^\rho(s)\int_{t_0}^s\psi_1(s,r)\check{X}(r)dW(r)+S(s)\Big|^2\Big>ds\bigg]^{\frac{1}{2}}\\
			\le&\ \frac{1}{4}\mathbb{E}\Big[\sup_{t\in[t_0,T]}|\check{X}(t)|^2\Big]+4\mathbb{E}\int_{t_0}^T\Big| I^\rho(s)\check{X}(s)+K^\rho(s)\check{X}(s-\delta)+L^\rho(s)\int_{t_0}^s\phi_1(s,r)\check{X}(r)dr\\
			&\qquad\qquad+M^\rho(s)\int_{t_0}^s\psi_1(s,r)\check{X}(r)dW(r)+S(s)\Big|^2ds\\
			\le&\ \frac{1}{4}\mathbb{E}\Big[\sup_{t\in[t_0,T]}|\check{X}(t)|^2\Big]+C(2,\bar{L},T)\int_{t_0}^T\mathbb{E}\Big[\sup_{t\in[t_0,T]}\big|\check{X}(t)\big|^2\Big]dt+C(2,\bar{L},T)\mathbb{E}\int_{t_0}^T|S(t)|^2dt.
	\end{aligned}\end{equation}
	Inserting (\ref{esphix2})-(\ref{esdiffx2}) into (\ref{esx2}), we deduce
	\begin{equation}\begin{aligned}
			\mathbb{E}\Big[\sup_{t\in[t_0,T]}|\check{X}(t)|^2\Big]\le&\ C(2,\bar{L},T)\int_{t_0}^T\mathbb{E}\Big[\sup_{t\in[t_0,T]}\big|\check{X}(t)\big|^2\Big]dt\\
			&+C(2,\bar{L},T)\mathbb{E}\int_{t_0}^T|S(t)|^2dt+C(2,\bar{L},T)\mathbb{E}\int_{t_0}^T|R(t)|^2dt,	
	\end{aligned}\end{equation}
	where $C(2,\bar{L},T)$ is a finite constant. Further, combine with the fact 
	that $\mathbb{E}\big[|S(t)|^2\big]$, $\mathbb{E}\big[|R(t)|^2\big]\to 0$ as $\rho\to0$ for all $t\in [t_0,T]$, one has by Gronwall's inequality,
	\begin{equation*}
		\lim_{\rho\to 0}\mathbb{E}\Big[\sup_{t\in[t_0,T]}|\check{X}(t)|^2\Big]=0.
	\end{equation*}
	This proof is completed.
\end{proof}

By the optimality of $u^{\ast}(\cdot)$ for the stochastic controlled system with extended mixed delays, it follows that
\begin{equation*}
	\frac{1}{\rho}\big[J(u^\rho(\cdot))-J(u^\ast(\cdot))\big]\le0.
\end{equation*}
From this, we have the following lemma, whose derivation is classical and we omit it.
\begin{mylem}
	Suppose that \textbf{(H3.1)} holds. Let $(x^\ast(\cdot),u^\ast(\cdot))$ be an optimal pair, and $x^\rho(\cdot)$ denotes the state trajectory corresponding to the control $u^\rho(\cdot)$ by (\ref{u rho}), then the following variational inequality holds:
	\begin{equation}\begin{aligned}\label{variational inequality}
			\mathbb{E}\bigg\{&\int_{t_0}^T\Big[l_x(t)\hat{x}(t)+l_y(t)\hat{y}(t)+l_z(t)\hat{z}(t)+l_{\kappa}(t)\hat{\kappa}(t)+l_u(t)v(t)+l_{\mu}(t)v(t-\delta)\\
			&\qquad+l_{\nu}(t)\int_{t_0}^t\phi_2(t,s)v(s)ds+l_{\lambda}(t)\int_{t_0}^t\psi_2(t,s)v(s)dW(s) \Big]dt\\
			&\qquad+h_x(T)\hat{x}(T)+h_y(T)\hat{y}(T)+h_z(T)\hat{z}(T)+h_{\kappa}(T)\hat{\kappa}(T) \bigg\}\le0.
	\end{aligned}\end{equation}
\end{mylem}
Next, we transform the variational equation SDDE (\ref{variational equation}) and variational inequality (\ref{variational inequality}) into SVIE. By Fubini's theorem,
\begin{equation}\begin{aligned}\label{z(t)tansform}
		\hat{z}(t)&=\int_{t_0}^t\phi_1(t,s)\hat{x}(s)ds\\
		&=\int_{t_0}^t\phi_1(t,s)\bigg\{\int_{t_0}^s\Big[b_x(r)\hat{x}(r)+b_y(r)\hat{y}(r)+b_z(r)\hat{z}(r)+b_{\kappa}(r)\hat{\kappa}(r)+b_u(r)v(r)+b_{\mu}(r)v(r-\delta)\\
		&\qquad+b_{\nu}(r)\int_{t_0}^r\phi_2(r,\tau)v(\tau)d\tau+b_{\lambda}(r)\int_{t_0}^r\psi_2(r,\tau)v(\tau)dW(\tau)\Big]dr\\
		&\quad+\int_{t_0}^s\Big[\sigma_x(r)\hat{x}(r)+\sigma_y(r)\hat{y}(r)+\sigma_z(r)\hat{z}(r)+\sigma_{\kappa}(r)\hat{\kappa}(r)+\sigma_u(r)v(r)+\sigma_{\mu}(r)v(r-\delta)\\
		&\qquad+\sigma_{\nu}(r)\int_{t_0}^r\phi_2(r,\tau)v(\tau)d\tau+\sigma_{\lambda}(r)\int_{t_0}^r\psi_2(r,\tau)v(\tau)dW(\tau)\Big]dW(r)\bigg\}ds\\
		&=\int_{t_0}^t\mathcal{E}_1(t,s)\Big[b_x(s)\hat{x}(s)+b_y(s)\hat{y}(s)+b_z(s)\hat{z}(s)+b_{\kappa}(s)\hat{\kappa}(s)+b_u(s)v(s)+b_{\mu}(s)v(s-\delta)\\
		&\qquad\qquad\qquad+b_{\nu}(s)\int_{t_0}^s\phi_2(s,\tau)v(\tau)d\tau+b_{\lambda}(s)\int_{t_0}^s\psi_2(s,\tau)v(\tau)dW(\tau)\Big]ds\\
		&\quad+\int_{t_0}^t\mathcal{E}_1(t,s)\Big[\sigma_x(s)\hat{x}(s)+\sigma_y(s)\hat{y}(s)+\sigma_z(s)\hat{z}(s)+\sigma_{\kappa}(s)\hat{\kappa}(s)+\sigma_u(s)v(s)+\sigma_{\mu}(s)v(s-\delta)\\
		&\qquad\qquad\qquad\quad+\sigma_{\nu}(s)\int_{t_0}^s\phi_2(s,\tau)v(\tau)d\tau+\sigma_{\lambda}(s)\int_{t_0}^s\psi_2(s,\tau)v(\tau)dW(\tau)\Big]dW(s),
\end{aligned}\end{equation}
where $\mathcal{E}_1(t,s):=\int_s^t\phi_1(t,r)dr$.

Similarly, by stochastic Fubini's theorem, we deduce
\begin{equation*}\begin{aligned}
		\hat{\kappa}(t)&=\int_{t_0}^t\psi_1(t,s)\hat{x}(s)dW(s)\\
		&=\int_{t_0}^t\mathcal{E}_2(t,s)\Big[b_x(s)\hat{x}(s)+b_y(s)\hat{y}(s)+b_z(s)\hat{z}(s)+b_{\kappa}(s)\hat{\kappa}(s)+b_u(s)v(s)\\
		&\qquad\qquad+b_{\mu}(s)v(s-\delta)+b_{\nu}(s)\int_{t_0}^s\phi_2(s,\tau)v(\tau)d\tau+b_{\lambda}(s)\int_{t_0}^s\psi_2(s,\tau)v(\tau)dW(\tau)\Big]ds\\
\end{aligned}\end{equation*}
\begin{equation}\begin{aligned}\label{kappa(t)transform}
		&\quad+\int_{t_0}^t\mathcal{E}_2(t,s)\Big[\sigma_x(s)\hat{x}(s)+\sigma_y(s)\hat{y}(s)+\sigma_z(s)\hat{z}(s)+\sigma_{\kappa}(s)\hat{\kappa}(s)+\sigma_u(s)v(s)\\
		&\qquad\quad+\sigma_{\mu}(s)v(s-\delta)+\sigma_{\nu}(s)\int_{t_0}^s\phi_2(s,\tau)v(\tau)d\tau+\sigma_{\lambda}(s)\int_{t_0}^s\psi_2(s,\tau)v(\tau)dW(\tau)\Big]dW(s),
\end{aligned}\end{equation}
where $\mathcal{E}_2(t,s):=\int_s^t\psi_1(t,r)dW(r)$.

Define
\begin{equation}\label{X(t)}
	X(t)^\top:=\begin{bmatrix}
		\hat{x}(t)^\top &
		\hat{y}(t)^\top &
		\hat{z}(t)^\top &
		\hat{\kappa}(t)^\top
	\end{bmatrix},
\end{equation}
and
\small
\begin{equation*}\begin{aligned}
		\mathbb{A}(t,s):&=\begin{bmatrix}
			b_x(s) & b_y(s) & b_z(s) & b_{\kappa}(s)\\
			\mathbb{I}_{(\delta,\infty)}(t-s)b_x(s) & \mathbb{I}_{(\delta,\infty)}(t-s)b_y(s) & \mathbb{I}_{(\delta,\infty)}(t-s)b_z(s) & \mathbb{I}_{(\delta,\infty)}(t-s)b_{\kappa}(s)\\
			\mathcal{E}_1(t,s)b_x(s) & \mathcal{E}_1(t,s)b_y(s) & \mathcal{E}_1(t,s)b_z(s) & \mathcal{E}_1(t,s)b_{\kappa}(s)\\
			\mathcal{E}_2(t,s)b_x(s) & \mathcal{E}_2(t,s)b_y(s) & \mathcal{E}_2(t,s)b_z(s) & \mathcal{E}_2(t,s)b_{\kappa}(s)
		\end{bmatrix},\\
		\mathbb{C}(t,s):&=\begin{bmatrix}
			\sigma_x(s) & \sigma_y(s) & \sigma_z(s) & \sigma_{\kappa}(s)\\
			\mathbb{I}_{(\delta,\infty)}(t-s)\sigma_x(s) & \mathbb{I}_{(\delta,\infty)}(t-s)\sigma_y(s) & \mathbb{I}_{(\delta,\infty)}(t-s)\sigma_z(s) & \mathbb{I}_{(\delta,\infty)}(t-s)\sigma_{\kappa}(s)\\
			\mathcal{E}_1(t,s)\sigma_x(s) & \mathcal{E}_1(t,s)\sigma_y(s) & \mathcal{E}_1(t,s)\sigma_z(s) & \mathcal{E}_1(t,s)\sigma_{\kappa}(s)\\
			\mathcal{E}_2(t,s)\sigma_x(s) & \mathcal{E}_2(t,s)\sigma_y(s) & \mathcal{E}_2(t,s)\sigma_z(s) & \mathcal{E}_2(t,s)\sigma_{\kappa}(s)
		\end{bmatrix},\\
		\mathbb{A}_1(t,s):&=\begin{bmatrix}
			b_x(s) & b_y(s) & b_z(s) & b_{\kappa}(s)\\
			\mathbb{I}_{(\delta,\infty)}(t-s)b_x(s) & \mathbb{I}_{(\delta,\infty)}(t-s)b_y(s) & \mathbb{I}_{(\delta,\infty)}(t-s)b_z(s) & \mathbb{I}_{(\delta,\infty)}(t-s)b_{\kappa}(s)\\
			\mathcal{E}_1(t,s)b_x(s) & \mathcal{E}_1(t,s)b_y(s) & \mathcal{E}_1(t,s)b_z(s) & \mathcal{E}_1(t,s)b_{\kappa}(s)\\
			0 & 0 & 0 & 0
		\end{bmatrix},\\
		\mathbb{C}_1(t,s):&=\begin{bmatrix}
			\sigma_x(s) & \sigma_y(s) & \sigma_z(s) & \sigma_{\kappa}(s)\\
			\mathbb{I}_{(\delta,\infty)}(t-s)\sigma_x(s) & \mathbb{I}_{(\delta,\infty)}(t-s)\sigma_y(s) & \mathbb{I}_{(\delta,\infty)}(t-s)\sigma_z(s) & \mathbb{I}_{(\delta,\infty)}(t-s)\sigma_{\kappa}(s)\\
			\mathcal{E}_1(t,s)\sigma_x(s) & \mathcal{E}_1(t,s)\sigma_y(s) & \mathcal{E}_1(t,s)\sigma_z(s) & \mathcal{E}_1(t,s)\sigma_{\kappa}(s)\\
			0 & 0 & 0 & 0
		\end{bmatrix},\\
		\mathbb{A}_2(t,s):&=\begin{bmatrix}
			0 & 0 & 0 & 0\\
			0 & 0 & 0 & 0\\
			0 & 0 & 0 & 0\\
			\mathcal{E}_2(t,s)b_x(s) & \mathcal{E}_2(t,s)b_y(s) & \mathcal{E}_2(t,s)b_z(s) & \mathcal{E}_2(t,s)b_{\kappa}(s)
		\end{bmatrix},\\
		\mathbb{C}_2(t,s):&=\begin{bmatrix}
			0 & 0 & 0 & 0\\
			0 & 0 & 0 & 0\\
			0 & 0 & 0 & 0\\
			\mathcal{E}_2(t,s)\sigma_x(s) & \mathcal{E}_2(t,s)\sigma_y(s) & \mathcal{E}_2(t,s)\sigma_z(s) & \mathcal{E}_2(t,s)\sigma_{\kappa}(s)
		\end{bmatrix},\\
		\hat{\mathbb{A}}_2(t):&=\begin{bmatrix}
			0 & 0 & 0 & 0\\
			0 & 0 & 0 & 0\\ 
			0 & 0 & 0 & 0\\
			b_x(t) & b_y(t) & b_z(t) & b_{\kappa}(t)
		\end{bmatrix}, \quad 
		\hat{\mathbb{C}}_2(t):=\begin{bmatrix}
			0 & 0 & 0 & 0\\
			0 & 0 & 0 & 0\\ 
			0 & 0 & 0 & 0\\
			\sigma_x(t) & \sigma_y(t) & \sigma_z(t) & \sigma_{\kappa}(t)
		\end{bmatrix},\\
\end{aligned}\end{equation*}	
\begin{equation*}\begin{aligned}\label{coefficient1}
		\mathbb{B}(t,s):&=\begin{bmatrix}
			\Xi_b(s) \\
			\mathbb{I}_{(\delta,\infty)}(t-s)\Xi_b(s) \\
			\mathcal{E}_1(t,s)\Xi_b(s) \\
			\mathcal{E}_2(t,s)\Xi_b(s) 
		\end{bmatrix},\
		\mathbb{D}(t,s):=\begin{bmatrix}
			\Xi_{\sigma}(s) \\
			\mathbb{I}_{(\delta,\infty)}(t-s)\Xi_{\sigma}(s) \\
			\mathcal{E}_1(t,s)\Xi_{\sigma}(s) \\
			\mathcal{E}_2(t,s)\Xi_{\sigma}(s) 
		\end{bmatrix},\
		\mathbb{B}_2(t,s):=\begin{bmatrix}
			0\\ 0\\ 0\\ \mathcal{E}_2(t,s)\Xi_b(s) 
		\end{bmatrix},\\
		\mathbb{B}_1(t,s):&=\begin{bmatrix}
			\Xi_b(s) \\
			\mathbb{I}_{(\delta,\infty)}(t-s)\Xi_b(s) \\
			\mathcal{E}_1(t,s)\Xi_b(s) \\
			0
		\end{bmatrix},\
		\mathbb{D}_1(t,s):=\begin{bmatrix}
			\Xi_{\sigma}(s) \\
			\mathbb{I}_{(\delta,\infty)}(t-s)\Xi_{\sigma}(s) \\
			\mathcal{E}_1(t,s)\Xi_{\sigma}(s) \\
			0 
		\end{bmatrix},
		\mathbb{D}_2(t,s):=\begin{bmatrix}
			0\\ 0\\ 0\\
			\mathcal{E}_2(t,s)\Xi_{\sigma}(s) 
		\end{bmatrix},
\end{aligned}\end{equation*}
\normalsize
with
\begin{equation*}
	\Xi_f(t):=f_u(t)v(t)+f_{\mu}(t)v(t-\delta)+f_{\nu}(t)\int_{t_0}^t\phi_2(t,s)v(s)ds+f_{\lambda}(t)\int_{t_0}^t\psi_2(t,s)v(s)dW(s),\ f=b,\sigma,l.
\end{equation*}

Then, base on (\ref{variational equation}), (\ref{z(t)tansform}) and (\ref{kappa(t)transform}), $X(\cdot)$ satisfies the following SVIE:
\begin{equation}\begin{aligned}\label{FSVIE}
		X(t)&=\int_{t_0}^t\big[\mathbb{A}(t,s)X(s)+\mathbb{B}(t,s)\big]ds+\int_{t_0}^t\big[\mathbb{C}(t,s)X(s)+\mathbb{D}(t,s)\big]dW(s)\\
		:&=\varphi(t)+\int_{t_0}^t\mathbb{A}(t,s)X(s)ds+\int_{t_0}^t\mathbb{C}(t,s)X(s)dW(s):=\hat{X}(t)+\tilde{X}(t),
\end{aligned}\end{equation}
where
\begin{equation*}\begin{aligned}
		\varphi(t):&=\int_{t_0}^t\mathbb{B}(t,s)ds+\int_{t_0}^t\mathbb{D}(t,s)dW(s):=\hat{\varphi}(t)+\tilde{\varphi}(t),\\
		\hat{X}(t):
		&=\hat{\varphi}(t)+\int_{t_0}^t\mathbb{A}_1(t,s)X(s)ds+\int_{t_0}^t\mathbb{C}_1(t,s)X(s)dW(s),\\
		\tilde{X}(t):
		&=\tilde{\varphi}(t)+\int_{t_0}^t\mathbb{A}_2(t,s)X(s)ds+\int_{t_0}^t\mathbb{C}_2(t,s)X(s)dW(s),\\
\end{aligned}\end{equation*}
and 
\begin{equation*}
	\hat{\varphi}(t):=\int_{t_0}^t\mathbb{B}_1(t,s)ds+\int_{t_0}^t\mathbb{D}_1(t,s)dW(s),\quad
	\tilde{\varphi}(t):=\int_{t_0}^t\mathbb{B}_2(t,s)ds+\int_{t_0}^t\mathbb{D}_2(t,s)dW(s).
\end{equation*}

By Proposition 2.1 in \cite{WY23} and the assumption \textbf{(H3.1)}, (\ref{FSVIE}) admits unique solution. Therefore, the above variatioanl inequality (\ref{variational inequality}) can be converted into
\begin{equation}\label{variational inequality-1}
	\frac{1}{\rho}\big[J\big(u^{\rho}(\cdot)\big)-J\big(u^{\ast}(\cdot)\big)\big]\le\mathbb{E}\bigg\{\int_{t_0}^T\big[\mathbb{L}(t)X(t)+\Xi_l(t)\big]dt+\mathbb{H}X(T)\bigg\}\le0,
\end{equation}
where 
\begin{equation}\begin{aligned}\label{cost coefficient3}
		\mathbb{L}(t):&=\begin{bmatrix}
			l_x(t) & l_y(t) & l_z(t) & l_{\kappa}(t) 
		\end{bmatrix},\quad
		\mathbb{H}:=\begin{bmatrix}
			h_x(T) & h_y(T) & h_z(T) & h_{\kappa}(T) 
		\end{bmatrix},\\
		\mathbb{H}_1:&=\begin{bmatrix}
			h_x(T) & h_y(T) & h_z(T) & 0
		\end{bmatrix},\quad
		\mathbb{H}_2:=\begin{bmatrix}
			0 & 0 & 0 & h_{\kappa}(T) 
		\end{bmatrix}.	
\end{aligned}\end{equation}

\section{Stochastic optimal control problem with extended mixed delays}

We treat the term about $X(\cdot)$ in (\ref{FSVIE}). Inspired by \cite{YJM08}, we introduce the adjoint equation as follows:
\begin{equation}\left\{\begin{aligned}\label{adjoint equation}
		(a)\ \eta(t)&=\mathbb{H}^\top-\int_t^T\zeta(s)dW(s),\quad t\in[t_0,T],\\
		(b)\ Y(t)&=\mathbb{L}(t)^\top+\mathbb{A}_1(T,t)^\top\mathbb{H} ^\top+\mathbb{C}_1(T,t)^\top \zeta(t)\\
		&\quad+\int_t^T\Big\{\hat{\mathbb{A}}_2(t)^\top\psi_1(T,s)^\top\mathbb{E}^{\mathcal{F}_s}\big[D_s\mathbb{H}^\top\big] +\hat{\mathbb{C}}_2(t)^\top \mathbb{E}^{\mathcal{F}_t}\big[D_t\big(\psi_1(T,s)^\top\zeta(s)\big)\big]\Big\}ds\\
		&\quad+\int_t^T\bigg\{\mathbb{A}_1(s,t)^\top Y(s)+\mathbb{C}_1(s,t)^\top Z(s,t)+\int_t^s\Big\{\hat{\mathbb{A}}_2(t)^\top\psi_1(s,r)^\top\\
		&\quad\qquad \times \mathbb{E}^{\mathcal{F}_r}\big[D_rY(s)\big]+\hat{\mathbb{C}}_2(t)^\top \mathbb{E}^{\mathcal{F}_t}\big[D_t\big(\psi_1(s,r)^\top Z(s,r)\big)\big]\Big\}dr\bigg\}ds\\
		&\quad-\int_t^TZ(t,s)dW(s),\quad t\in[t_0,T],\\
		(c)\ Y(t)&=\mathbb{E}^{\mathcal{F}_{t_0}}\big[Y(t)\big]+\int_{t_0}^tZ(t,s)dW(s),\quad t\in[t_0,T].\\
	\end{aligned}\right.\end{equation}

\begin{Remark}
	By Theorem 4.1 in \cite{PP90}, BSDE (4.1a) admits a unique adapted solution $(\eta(\cdot),\zeta(\cdot))$. On the other hand, equation (4.1b) is a linear BSVIE with Malliavin derivatives, especially the term $D_t\big(\psi_1(s,r)^\top Z(s,r)\big)$, where the presence of this successfully overcome challenges induced by the extended noisy memory, but gives rise to substantial difficulties in establishing the existence and uniqueness of an adapted M-solution satisfying (4.1c) under the assumption \textbf{(H3.1)}. In this paper, we therefore assume that (4.1b) admits a unique adapted M-solution satisfying (4.1c). However, in the special case where \textcolor{blue}{$\psi_1(\cdot,\cdot)\equiv 0$}, equation (4.1b) indeed admits a unique adapted M-solution satisfying (4.1c) under the assumption \textbf{(H3.1)} by the Proposition~\ref{BSVIE exist}.
	
\end{Remark}

To facilitate the proof, we denote
\begin{equation}\begin{aligned}
		\psi(t):=&\ \mathbb{L}(t)^\top+\mathbb{A}_1(T,t)^\top\mathbb{H} ^\top+\mathbb{C}_1(T,t)^\top \zeta(t),\\	
		&+\int_t^T\Big\{\hat{\mathbb{A}}_2(t)^\top\psi_1(T,s)^\top\mathbb{E}^{\mathcal{F}_s}\big[D_s\mathbb{H}^\top\big] +\hat{\mathbb{C}}_2(t)^\top \mathbb{E}^{\mathcal{F}_t}\big[D_t\big(\psi_1(T,s)^\top\zeta(s)\big)\big]\Big\}ds.\\
\end{aligned}\end{equation}
Combine with the definition of (\ref{X(t)}), and let
\begin{equation}\label{adjointeqvector}
	Y(t):=\begin{bmatrix}
		Y^1(t)\\
		Y^2(t)\\
		Y^3(t)\\
		Y^4(t)
	\end{bmatrix}, \quad 
	Z(t,s):=\begin{bmatrix}
		Z^1(t,s)\\
		Z^2(t,s)\\
		Z^3(t,s)\\
		Z^4(t,s)
	\end{bmatrix},\quad 
	\eta(t):=\begin{bmatrix}
		\eta^1(t)\\
		\eta^2(t)\\
		\eta^3(t)\\
		\eta^4(t)
	\end{bmatrix}, \quad	
	\zeta(t):=\begin{bmatrix}
		\zeta^1(t)\\
		\zeta^2(t)\\
		\zeta^3(t)\\
		\zeta^4(t)
	\end{bmatrix}.
\end{equation}

\subsection{Stochastic maximum principle}

The appearance of $\mathcal{E}_2(\cdot,\cdot)$ causes the coefficients of the SVIE (\ref{FSVIE}) do not satisfy the conditions required in Theorem 5.1 of \cite{YJM08}.  To overcome this difficulty, we employ a technique referred to as \emph{coefficient decomposition} in this paper, by which the adjoint equation (\ref{adjoint equation}) is introduced and stochastic maximum principle is derived.

\begin{mylem}
	Let \textbf{(H3.1)} be satisfied. Suppose that $(x^\ast(\cdot),u^\ast(\cdot))$ is an optimal pair to \textbf{Problem (P)}, $(\eta(\cdot),\zeta(\cdot),Y(\cdot),Z(\cdot,\cdot))$ is a solution to (\ref{adjoint equation}). Then, the variational inequality (\ref{variational inequality-1}) can be derived as follows
	\begin{equation}\begin{aligned}\label{variational inequality-2}
			&\rho^{-1}\big[J(u^{\rho}(\cdot))-J(u^{\ast}(\cdot)) \big]\\
			=&\ \mathbb{E}\int_{t_0}^T\bigg\{\Xi_l(s)+\Big<\Xi_b(s),\int_s^T Y^1(t)dt+\int_{s+\delta}^T Y^2(t)dt\mathbb{I}_{[t_0,T-\delta)}(s)+\int_s^T \mathcal{E}_1(t,s)^\top Y^3(t)dt\\
			&\quad+\int_s^T\int_s^t\psi_1(t,r)^\top\mathbb{E}^{\mathcal{F}_r}\big[D_rY^4(t)\big]drdt+h_x(T)^\top+h_y(T)^\top\mathbb{I}_{[t_0,T-\delta)}(s)+\mathcal{E}_1(T,s)^\top h_z(T)^\top\\
			&\quad+\int_s^T\psi_1(T,r)^\top\mathbb{E}^{\mathcal{F}_r}\big[D_rh_{\kappa}(T)^\top\big]dr \Big>+\Big<\Xi_{\sigma}(s),\int_s^T Z^1(t,s)dt+\int_{s+\delta}^T Z^2(t,s)dt\mathbb{I}_{[t_0,T-\delta)}(s)\\
			&\quad+\int_s^T \mathcal{E}_1(t,s)^\top Z^3(t,s)dt+\int_s^T\int_s^t\mathbb{E}^{\mathcal{F}_s}\big[D_s\big(\psi_1(t,r)^\top Z^4(t,r)\big)\big]drdt+\zeta^1(s)\\
			&\quad+\zeta^2(s)\mathbb{I}_{[t_0,T-\delta)}(s)+\mathcal{E}_1(T,s)^\top\zeta^3(s)+\int_s^T\mathbb{E}^{\mathcal{F}_s}\big[D_s\big(\psi_1(T,r)^\top\zeta^4(r)\big) \big]dr\bigg\}ds\le0.
	\end{aligned}\end{equation}
\end{mylem} 
\begin{proof}
	Inspired by \cite{YJM08}, recalled back SVIE (\ref{FSVIE}) and (\ref{adjoint equation}), by Fubini's theorem, we deduce
	\begin{equation}\begin{aligned}\label{varphiY}
			&\mathbb{E}\int_{t_0}^T\big<\varphi(t),Y(t)\big>dt\\
			=&\ \mathbb{E}\int_{t_0}^T\big<X(t),Y(t)\big>dt-\mathbb{E}\int_{t_0}^T\Big<\int_{t_0}^t\mathbb{A}_1(t,s)X(s)ds+\int_{t_0}^t\mathbb{C}_1(t,s)X(s)dW(s),Y(t)\Big>dt\\
			&\ -\mathbb{E}\int_{t_0}^T\Big<\int_{t_0}^t\mathbb{A}_2(t,s)X(s)ds+\int_{t_0}^t\mathbb{C}_2(t,s)X(s)dW(s),Y(t)\Big>dt\\
			=&\ \mathbb{E}\int_{t_0}^T\big<X(t),Y(t)\big>dt-\mathbb{E}\int_{t_0}^T\Big<\int_{t}^T\big[\mathbb{A}_1(s,t)^\top Y(s)+\mathbb{C}_1(s,t)^\top Z(s,t)\big]ds,X(t)\Big>dt\\
			&\ -\mathbb{E}\int_{t_0}^T\Big<\int_{t_0}^t\mathbb{A}_2(t,s)X(s)ds+\int_{t_0}^t\mathbb{C}_2(t,s)X(s)dW(s),Y(t)\Big>dt.
	\end{aligned}\end{equation}
	Moreover, by stochastic Fubini's theorem and generalized duality principle, we derive
	\small\begin{equation*}
		\begin{aligned}
			&\mathbb{E}\int_{t_0}^T\int_{t_0}^t\big<\mathbb{A}_2(t,s)X(s),Y(t)\big>dsdt\\
			=&\ \mathbb{E}\int_{t_0}^T\int_{t_0}^t\Big<Y^3(t),\int_s^t\psi_1(t,r)dW(r)\big[b_x(s)\hat{x}(s)+b_y(s)\hat{y}(s)+b_z(s)\hat{z}(s)+b_{\kappa}(s)\hat{\kappa}(s)\big] \Big>dsdt\\
		\end{aligned}
	\end{equation*}
	\begin{equation}\label{Malliavin on A2}
		\begin{aligned}
			=&\ \mathbb{E}\int_{t_0}^T\int_{t_0}^t\int_{t_0}^r\Big<Y^3(t),\psi_1(t,r)\big[b_x(s)\hat{x}(s)+b_y(s)\hat{y}(s)+b_z(s)\hat{z}(s)+b_{\kappa}(s)\hat{\kappa}(s)\big] \Big>dsdW(r)dt\\	
			=&\ \mathbb{E}\int_{t_0}^T\Big<Y^3(t),\int_{t_0}^t\Big[\psi_1(t,s)\int_{t_0}^s\big[b_x(r)\hat{x}(r)+b_y(r)\hat{y}(r)+b_z(r)\hat{z}(r)+b_{\kappa}(r)\hat{\kappa}(r)\big]dr\Big] dW(s)\Big>dt\\
			=&\ \mathbb{E}\int_{t_0}^T\int_{t_0}^t\Big<\mathbb{E}^{\mathcal{F}_s}\big[D_sY^3(t)\big],\psi_1(t,s)\int_{t_0}^s\big[b_x(r)\hat{x}(r)+b_y(r)\hat{y}(r)+b_z(r)\hat{z}(r)+b_{\kappa}(r)\hat{\kappa}(r)\big]dr\Big>dsdt\\
			=&\ \mathbb{E}\int_{t_0}^T\int_{t_0}^t\int_r^t\Big<\psi_1(t,s)^\top\mathbb{E}^{\mathcal{F}_s}\big[D_sY^3(t)\big],b_x(r)\hat{x}(r)+b_y(r)\hat{y}(r)+b_z(r)\hat{z}(r)+b_{\kappa}(r)\hat{\kappa}(r)\Big>dsdrdt\\
			=&\ \mathbb{E}\int_{t_0}^T\int_{t_0}^t\Big<\int_s^t\psi_1(t,r)^\top\mathbb{E}^{\mathcal{F}_r}\big[D_rY^3(t)\big]dr,b_x(s)\hat{x}(s)+b_y(s)\hat{y}(s)+b_z(s)\hat{z}(s)+b_{\kappa}(s)\hat{\kappa}(s)\Big>dsdt\\
			=&\ \mathbb{E}\int_{t_0}^T\int_s^T\Big<\int_s^t\psi_1(t,r)^\top\mathbb{E}^{\mathcal{F}_r}\big[D_rY^3(t)\big]dr,b_x(s)\hat{x}(s)+b_y(s)\hat{y}(s)+b_z(s)\hat{z}(s)+b_{\kappa}(s)\hat{\kappa}(s)\Big>dtds\\
			=&\ \mathbb{E}\int_{t_0}^T\Big<\int_t^T\int_t^s\psi_1(s,r)^\top\mathbb{E}^{\mathcal{F}_r}\big[D_rY^3(s)\big]drds,b_x(t)\hat{x}(t)+b_y(t)\hat{y}(t)+b_z(t)\hat{z}(t)+b_{\kappa}(t)\hat{\kappa}(t)\Big>dt\\
			=&\ \mathbb{E}\int_{t_0}^T\Big<\hat{\mathbb{A}}_2(t)^\top\int_t^T\int_t^s\psi_1(s,r)^\top\mathbb{E}^{\mathcal{F}_r}\big[D_rY^3(s)\big]drds,X(t)\Big>dt.
		\end{aligned}
	\end{equation}
	\normalsize
	With the (4.16) of Theorem 4.1 in \cite{YJM08}, one has
	\small
	\begin{equation}\label{Malliavin on C2}
		\begin{aligned}
			&\mathbb{E}\int_{t_0}^T\int_{t_0}^t\big<\mathbb{C}_2(t,s)X(s),Y(t)\big>dW(s)dt\\
			=&\ \mathbb{E}\int_{t_0}^T\int_{t_0}^t\Big<Y^3(t),\int_s^t\psi_1(t,r)dW(r)\big[\sigma_x(s)\hat{x}(s)+\sigma_y(s)\hat{y}(s)+\sigma_z(s)\hat{z}(s)+\sigma_{\kappa}(s)\hat{\kappa}(s)\big] \Big>dW(s)dt\\
			=&\ \mathbb{E}\int_{t_0}^T\Big<Y^3(t),\int_{t_0}^t\Big[\psi_1(t,s)\int_{t_0}^s\big[\sigma_x(r)\hat{x}(r)+\sigma_y(r)\hat{y}(r)+\sigma_z(r)\hat{z}(r)+\sigma_{\kappa}(r)\hat{\kappa}(r)\big]dW(r)\Big] dW(s)\Big>dt\\
			=&\ \mathbb{E}\int_{t_0}^T\int_{t_0}^t\Big<\mathbb{E}^{\mathcal{F}_s}\big[D_sY^3(t)\big],\psi_1(t,s)\int_{t_0}^s\big[\sigma_x(r)\hat{x}(r)+\sigma_y(r)\hat{y}(r)+\sigma_z(r)\hat{z}(r)+\sigma_{\kappa}(r)\hat{\kappa}(r)\big]dW(r)\Big>dsdt\\
			=&\ \int_{t_0}^T\int_{t_0}^t\mathbb{E}\Big<\psi_1(t,s)^\top Z^3(t,s),\int_{t_0}^s\big[\sigma_x(r)\hat{x}(r)+\sigma_y(r)\hat{y}(r)+\sigma_z(r)\hat{z}(r)+\sigma_{\kappa}(r)\hat{\kappa}(r)\big]dW(r)\Big>dsdt\\
			=&\ \mathbb{E}\int_{t_0}^T\int_{t_0}^t\int_{t_0}^s\Big<\mathbb{E}^{\mathcal{F}_r}\big[D_r\big(\psi_1(t,s)^\top Z^3(t,s)\big)\big],\sigma_x(r)\hat{x}(r)+\sigma_y(r)\hat{y}(r)+\sigma_z(r)\hat{z}(r)+\sigma_{\kappa}(r)\hat{\kappa}(r)\Big>drdsdt\\
			=&\ \mathbb{E}\int_{t_0}^T\int_{t_0}^t\Big<\int_s^t\mathbb{E}^{\mathcal{F}_s}\big[D_s\big(\psi_1(t,r)^\top Z^3(t,r)\big)\big]dr,\sigma_x(s)\hat{x}(s)+\sigma_y(s)\hat{y}(s)+\sigma_z(s)\hat{z}(s)+\sigma_{\kappa}(s)\hat{\kappa}(s)\Big>dsdt\\
			=&\ \mathbb{E}\int_{t_0}^T\Big<\int_t^T\int_t^s\mathbb{E}^{\mathcal{F}_t}\big[D_t\big(\psi_1(s,r)^\top Z^3(s,r)\big)\big]drds,\sigma_x(t)\hat{x}(t)+\sigma_y(t)\hat{y}(t)+\sigma_z(t)\hat{z}(t)+\sigma_{\kappa}(t)\hat{\kappa}(t)\Big>dt\\
			=&\ \mathbb{E}\int_{t_0}^T\Big<\hat{\mathbb{C}}_2(t)^\top\int_t^T\int_t^s\mathbb{E}^{\mathcal{F}_t}\big[D_t\big(\psi_1(s,r)^\top Z^3(s,r)\big)\big]drds,X(t)\Big>dt.
		\end{aligned}
	\end{equation}
	\normalsize
	Substituting (\ref{Malliavin on A2}) and (\ref{Malliavin on C2}) into the (\ref{varphiY}) yields
	\begin{equation*}
		\begin{aligned}
			&\mathbb{E}\int_{t_0}^T\big<\varphi(t),Y(t)\big>dt\\
			=&\ \mathbb{E}\int_{t_0}^T\big<X(t),Y(t)\big>dt-\mathbb{E}\int_{t_0}^T\Big<\int_{t}^T\big[\mathbb{A}_1(s,t)^\top Y(s)+\mathbb{C}_1(s,t)^\top Z(s,t)\big]ds,X(t)\Big>dt\\
		\end{aligned}
	\end{equation*}
	\begin{equation}\label{varphiY-1}
		\begin{aligned}
			&-\mathbb{E}\int_{t_0}^T\Big<\hat{\mathbb{A}}_2(t)^\top\int_t^T\int_t^s\psi_1(s,r)^\top\mathbb{E}^{\mathcal{F}_r}\big[D_rY^3(s)\big]drds,X(t)\Big>dt\\
			&-\mathbb{E}\int_{t_0}^T\Big<\hat{\mathbb{C}}_2(t)^\top\int_t^T\int_t^s\mathbb{E}^{\mathcal{F}_t}\big[D_t\big(\psi_1(s,r)^\top Z^3(s,r)\big)\big]drds,X(t)\Big>dt\\
			=&\ \mathbb{E}\int_{t_0}^T\big<\psi(t),X(t)\big>dt.
		\end{aligned}
	\end{equation}
	
	In the meanwhile,
	\begin{equation}\begin{aligned}\label{HtimesX}
			&\mathbb{E}\big[\big\langle\mathbb{H}, X(T)\big\rangle\big]=\mathbb{E}\Big[\Big<\mathbb{H}^\top,\varphi(T)+\int_{t_0}^T\mathbb{A}_1(T,t)X(t)dt+\int_{t_0}^T\mathbb{C}_1(T,t)X(t)dW(t)\\
			&\quad +\int_{t_0}^T\mathbb{A}_2(T,t)X(t)dt+\int_{t_0}^T\mathbb{C}_2(T,t)X(t)dW(t)\Big>\Big]\\
			=&\ \mathbb{E}\Big[\big<\mathbb{H}^\top,\varphi(T)\big>+\int_{t_0}^T\big<X(t),\mathbb{A}_1(T,t)^\top\mathbb{H}^\top+\mathbb{C}_1(T,t)^\top\zeta(t)\big>dt\\
			&\quad+\Big<\mathbb{H}^\top,\int_{t_0}^T\mathbb{A}_2(T,t)X(t)dt\Big>+\Big<\mathbb{H}^\top,\int_{t_0}^T\mathbb{C}_2(T,t)X(t)dW(t)\Big>\Big].\\ 
	\end{aligned}\end{equation}
	Similar to the treatment of (\ref{Malliavin on A2}) and (\ref{Malliavin on C2}), respectively, the following results are obtained:
	\small
	\begin{equation}\label{Malliavin on HA2}
		\begin{aligned}
			&\mathbb{E}\Big[\Big<\mathbb{H}^\top,\int_{t_0}^T\mathbb{A}_2(T,t)X(t)dt\Big>\Big]\\
			=&\ \mathbb{E}\Big[\Big<h_{\kappa}(T)^\top,\int_{t_0}^T\Big[\int_t^T\psi_1(T,s)dW(s)\big[b_x(t)\hat{x}(t)+b_y(t)\hat{y}(t)+b_z(t)\hat{z}(t)+b_{\kappa}(t)\hat{\kappa}(t)\big]\Big]dt \Big>\Big]\\
			=&\ \mathbb{E}\Big[\Big<h_{\kappa}(T)^\top,\int_{t_0}^T\Big[\psi_1(T,t)\int_{t_0}^t\big[b_x(s)\hat{x}(s)+b_y(s)\hat{y}(s)+b_z(s)\hat{z}(s)+b_{\kappa}(s)\hat{\kappa}(s)\big]ds\Big]dW(t)\Big>\Big]\\
			=&\ \mathbb{E}\int_{t_0}^T\Big<\mathbb{E}^{\mathcal{F}_t}\big[D_th_{\kappa}(T)^\top\big],\psi_1(T,t)\int_{t_0}^t\big[b_x(s)\hat{x}(s)+b_y(s)\hat{y}(s)+b_z(s)\hat{z}(s)+b_{\kappa}(s)\hat{\kappa}(s)\big]ds\Big>dt\\
			=&\ \mathbb{E}\int_{t_0}^T\Big<\int_t^T\psi_1(T,s)^\top\mathbb{E}^{\mathcal{F}_s}\big[D_sh_{\kappa}(T)^\top\big]ds,b_x(t)\hat{x}(t)+b_y(t)\hat{y}(t)+b_z(t)\hat{z}(t)+b_{\kappa}(t)\hat{\kappa}(t)\Big>dt\\
			=&\ \mathbb{E}\int_{t_0}^T\Big<\hat{\mathbb{A}}_2(t)^\top\int_t^T\psi_1(T,s)^\top\mathbb{E}^{\mathcal{F}_s}\big[D_s\mathbb{H}^\top\big]ds,X(t)\Big>dt,
		\end{aligned}
	\end{equation}
	\normalsize
	and
	\small
	\begin{equation*}
		\begin{aligned}
			&\mathbb{E}\Big[\Big<\mathbb{H}^\top,\int_{t_0}^T\mathbb{C}_2(T,t)X(t)dt\Big>\Big]\\
			=&\ \mathbb{E}\Big[\Big<h_{\kappa}(T)^\top,\int_{t_0}^T\Big[\int_t^T\psi_1(T,s)dW(s)\big[\sigma_x(t)\hat{x}(t)+\sigma_y(t)\hat{y}(t)+\sigma_z(t)\hat{z}(t)+\sigma_{\kappa}(t)\hat{\kappa}(t)\big]\Big]dW(t) \Big>\Big]\\
			=&\ \mathbb{E}\Big[\Big<h_{\kappa}(T)^\top,\int_{t_0}^T\Big[\psi_1(T,t)\int_{t_0}^t\big[\sigma_x(s)\hat{x}(s)+\sigma_y(s)\hat{y}(s)+\sigma_z(s)\hat{z}(s)+\sigma_{\kappa}(s)\hat{\kappa}(s)\big]dW(s)\Big]dW(t)\Big>\Big]\\
			=&\ \mathbb{E}\int_{t_0}^T\Big<\mathbb{E}^{\mathcal{F}_t}\big[D_th_{\kappa}(T)^\top\big],\psi_1(T,t)\int_{t_0}^t\big[\sigma_x(s)\hat{x}(s)+\sigma_y(s)\hat{y}(s)+\sigma_z(s)\hat{z}(s)+\sigma_{\kappa}(s)\hat{\kappa}(s)\big]dW(s)\Big>dt\\
			=&\ \int_{t_0}^T\mathbb{E}\Big<\psi_1(T,t)^\top\zeta^3(t),\int_{t_0}^t\big[\sigma_x(s)\hat{x}(s)+\sigma_y(s)\hat{y}(s)+\sigma_z(s)\hat{z}(s)+\sigma_{\kappa}(s)\hat{\kappa}(s)\big]dW(s)\Big>dt\\
		\end{aligned}
	\end{equation*}
	\begin{equation}\label{Malliavin on HC2}
		\begin{aligned}
			=&\ \mathbb{E}\int_{t_0}^T\int_{t_0}^t\Big<\mathbb{E}^{\mathcal{F}_s}\big[D_s\big(\psi_1(T,t)^\top\zeta^3(t)\big)\big],\sigma_x(s)\hat{x}(s)+\sigma_y(s)\hat{y}(s)+\sigma_z(s)\hat{z}(s)+\sigma_{\kappa}(s)\hat{\kappa}(s)\Big>dsdt\\
			=&\ \mathbb{E}\int_{t_0}^T\Big<\int_t^T\mathbb{E}^{\mathcal{F}_t}\big[D_t\big(\psi_1(T,s)^\top\zeta^3(s)\big)\big]ds,\sigma_x(t)\hat{x}(t)+\sigma_y(t)\hat{y}(t)+\sigma_z(t)\hat{z}(t)+\sigma_{\kappa}(t)\hat{\kappa}(t)\Big>dt\\
			=&\ \mathbb{E}\int_{t_0}^T\Big<\hat{\mathbb{C}}_2(t)^\top\int_t^T\mathbb{E}^{\mathcal{F}_t}\big[D_t\big(\psi_1(T,s)^\top\zeta^3(s)\big)\big]ds,X(t)\Big>dt.
		\end{aligned}
	\end{equation}
	\normalsize
	Taking (\ref{Malliavin on HA2}) and (\ref{Malliavin on HC2}) into (\ref{HtimesX}) with (\ref{varphiY}), results in
	\begin{equation}\label{dual principle}
		\mathbb{E}\big[\big\langle\mathbb{H}, X(T)\big\rangle\big]+\mathbb{E}\Big[\int_{t_0}^T\mathbb{L}(t)X(t)dt\Big]
		=\mathbb{E}\big[\big<\mathbb{H}^\top,\varphi(T)\big>\big]+\mathbb{E}\Big[\int_{t_0}^T\big<\varphi(t),Y(t)\big>dt\Big].
	\end{equation}
	
	Applying the coefficient decomposition again to (\ref{dual principle}), we gain  
	\begin{equation}\begin{aligned}\label{deal Terminal term}
			&\mathbb{E}\int_{t_0}^T\big<Y(t),\varphi(t)\big>dt+\mathbb{E}\big[\mathbb{H}\varphi(T)\big]=\mathbb{E}\int_{t_0}^T\int_{t_0}^t\big[\big<Y(t),\mathbb{B}_1(t,s)\big>+\big<Z(t,s),\mathbb{D}_1(t,s)\big>\big]dsdt\\
			&\quad +\mathbb{E}\int_{t_0}^T\Big<Y(t),\int_{t_0}^t\mathbb{B}_2(t,s)ds+ \int_{t_0}^t\mathbb{D}_2(t,s)dW(s)\Big>dt+\mathbb{E}\Big[\mathbb{H}\int_{t_0}^T\mathbb{B}_1(T,s)ds\\
			&\quad +\int_{t_0}^T\big<\zeta(s),\mathbb{D}_1(T,s)\big>ds\Big]
			 +\mathbb{E}\Big[\mathbb{H}\int_{t_0}^T\mathbb{B}_2(T,s)ds+\int_{t_0}^T\big<\zeta(s),\mathbb{D}_2(T,s)\big>ds\Big].
	\end{aligned}\end{equation}
	\normalsize
	Since
	\begin{equation*}\begin{aligned}
			&\mathbb{E}\int_{t_0}^T\int_{t_0}^t\big<Y(t),\mathbb{B}_2(t,s)\big>dsdt
			=\mathbb{E}\int_{t_0}^T\int_{t_0}^t\big<Y^3(t),\mathcal{E}_2(t,s)\Xi_b(s)\big>dsdt\\
			=&\ \mathbb{E}\int_{t_0}^T\Big<Y^3(t),\int_{t_0}^t\Big[\psi_1(t,s)\int_{t_0}^s\Xi_b(r)dr\Big]dW(s)\Big>dt\\
			=&\ \mathbb{E}\int_{t_0}^T\int_{t_0}^t\Big<\mathbb{E}^{\mathcal{F}_s}\big[D_sY^3(t)\big],\psi_1(t,s)\int_{t_0}^s\Xi_b(r)dr\Big>dsdt\\
			=&\ \mathbb{E}\int_{t_0}^T\int_{t_0}^t\Big<\int_s^t\psi_1(t,r)^\top\mathbb{E}^{\mathcal{F}_r}\big[D_rY^3(t)\big]dr,\Xi_b(s)\Big>dsdt\\
			=&\ \mathbb{E}\int_{t_0}^T\Big<\int_s^T\int_s^t\psi_1(t,r)^\top\mathbb{E}^{\mathcal{F}_r}\big[D_rY^3(t)\big]drdt,\Xi_b(s)\Big>ds,
	\end{aligned}\end{equation*}
	and
	\begin{equation*}\begin{aligned}
			&\mathbb{E}\int_{t_0}^T\int_{t_0}^t\big<Y(t),\mathbb{D}_2(t,s)\big>dW(s)dt
			=\mathbb{E}\int_{t_0}^T\int_{t_0}^t\big<Y^3(t),\mathcal{E}_2(t,s)\Xi_{\sigma}(s)\big>dW(s)dt\\
			=&\ \mathbb{E}\int_{t_0}^T\Big<Y^3(t),\int_{t_0}^t\Big[\psi_1(t,s)\int_{t_0}^s\Xi_{\sigma}(r)dW(r)\Big]dW(s)\Big>dt\\
			=&\ \mathbb{E}\int_{t_0}^T\int_{t_0}^t\Big<\mathbb{E}^{\mathcal{F}_s}\big[D_sY^3(t)\big],\psi_1(t,s)\int_{t_0}^s\Xi_{\sigma}(r)dW(r)\Big>dsdt\\ 
			=&\ \mathbb{E}\int_{t_0}^T\int_{t_0}^t\int_{t_0}^s\Big<\mathbb{E}^{\mathcal{F}_r}\big[D_r\big(\psi_1(t,s)^\top Z^3(t,s)\big)\big],\Xi_{\sigma}(r)\Big>drdsdt\\
			=&\ \mathbb{E}\int_{t_0}^T\int_{t_0}^t\Big<\int_s^t\mathbb{E}^{\mathcal{F}_s}\big[D_s\big(\psi_1(t,r)^\top Z^3(t,r)\big)\big]dr,\Xi_{\sigma}(s)\Big>dsdt\\
			=&\ \mathbb{E}\int_{t_0}^T\Big<\int_s^T\int_s^t\mathbb{E}^{\mathcal{F}_s}\big[D_s\big(\psi_1(t,r)^\top Z^3(t,r)\big)\big]drdt,\Xi_{\sigma}(s)\Big>ds,
	\end{aligned}\end{equation*}
	by Proposition \ref{pro2.5} and generalized duality principle, one has
	\begin{equation}\begin{aligned}\label{HB2}
			&\mathbb{E}\Big[\Big\langle\mathbb{H}^\top,\int_{t_0}^T\mathbb{B}_2(T,s)ds\Big\rangle\Big]
			=\mathbb{E}\Big[\Big<h_{\kappa}(T)^\top,\int_{t_0}^T\Big(\psi_1(T,s)\int_{t_0}^s\Xi_b(r)dr\Big)dW(s)\Big> \Big]\\
			=&\ \mathbb{E}\Big[\int_{t_0}^T\Big<\zeta^3(s),\psi_1(T,s)\int_{t_0}^s\Xi_b(r)dr\Big>ds\Big]
			=\mathbb{E}\int_{t_0}^T\Big<\int_s^T\psi_1(T,r)^\top\zeta^3(r)dr,\Xi_b(s)\Big>ds,
	\end{aligned}\end{equation}
	and
	\begin{equation}\begin{aligned}\label{HD2}
			&\mathbb{E}\Big[\mathbb{H}^\top,\int_{t_0}^T\mathbb{D}_2(T,s)dW(s)\Big]
			=\mathbb{E}\Big[\Big<h_{\kappa}(T)^\top,\int_{t_0}^T\Big(\psi_1(T,s)\int_{t_0}^s\Xi_{\sigma}(r)dW(r)\Big)dW(s)\Big> \Big]\\
			=&\ \mathbb{E}\Big[\int_{t_0}^T\Big<\zeta^3(s),\psi_1(T,s)\int_{t_0}^s\Xi_{\sigma}(r)dW(r)\Big>ds\Big]\\
			=&\ \mathbb{E}\int_{t_0}^T\int_{t_0}^s\Big<\mathbb{E}^{\mathcal{F}_r}\big[D_r\big(\psi_1(T,s)^\top\zeta^3(s)\big) \big],\Xi_{\sigma}(r) \Big>drds\\
			=&\ \mathbb{E}\int_{t_0}^T\Big<\int_s^T\mathbb{E}^{\mathcal{F}_s}\big[D_s\big(\psi_1(T,r)^\top\zeta^3(r)\big) \big]dr,\Xi_{\sigma}(s)\Big>ds.
	\end{aligned}\end{equation}
	Finally, substituting (\ref{dual principle})-(\ref{HD2}) into \eqref{variational inequality-1}, some calculations yields (\ref{variational inequality-2}) with \eqref{coefficient1}.
\end{proof}

In order to write the variational inequality (\ref{variational inequality-2}) in more concise form, and give the main result in this subsection. To this end, let us define
\begin{equation}\label{defp}
	\begin{aligned}
		p(t):&=\eta^1(t)+\eta^2(t)\mathbb{I}_{[t_0,T-\delta)}(t)+\mathcal{E}_1(T,t)^\top\eta^3(t)+\mathbb{E}^{\mathcal{F}_t}\Big[\int_t^T\psi_1(T,r)^\top D_r\eta^4(r)dr\Big]\\
		&\quad+\mathbb{E}^{\mathcal{F}_t}\Big[\int_t^T Y^1(s)ds+\int_{t+\delta}^T Y^2(s)ds\mathbb{I}_{[t_0,T-\delta)}(t)+\int_t^T \mathcal{E}_1(s,t)^\top Y^3(s)ds\\
		&\quad+\int_t^T\int_t^s\psi_1(s,r)^\top D_rY^4(s)drds\Big],\\
	\end{aligned}
\end{equation}
\begin{equation}\label{defq}
	\begin{aligned}
		q(t):&=\zeta^1(t)+\zeta^2(t)\mathbb{I}_{[t_0,T-\delta)}(t)+\mathcal{E}_1(T,t)^\top\zeta^3(t)+\int_t^T\mathbb{E}^{\mathcal{F}_t}\big[D_t\big(\psi_1(T,r)^\top\zeta^4(r)\big)\big]dr\\
		&\quad+\int_t^T Z^1(s,t)ds+\int_{t+\delta}^T Z^2(s,t)ds\mathbb{I}_{[t_0,T-\delta)}(t)+\int_t^T \mathcal{E}_1(s,t)^\top Z^3(s,t)ds\\
		&\quad+\int_t^T\int_t^s\mathbb{E}^{\mathcal{F}_t}\big[D_t\big(\psi_1(s,r)^\top Z^4(s,r)\big)\big]drds.
	\end{aligned}
\end{equation}

Now, we introduce the following Hamiltonian function $\mathcal{H}:\ [0,T]\times\mathbb{R}^n\times\mathbb{R}^n\times\mathbb{R}^n\times\mathbb{R}^n\times U\times U\times U\times U\times\mathbb{R}^n\times\mathbb{R}^n\to\mathbb{R}$ by
\begin{equation}\begin{aligned}\label{hamilton}
		\mathcal{H}(t,x,y,z,\kappa,p,q,u,\mu,\nu,\lambda)&:=l(t,x,y,z,\kappa,u,\mu,\nu,\lambda)+\langle b(t,x,y,z,\kappa,u,\mu,\nu,\lambda), p\rangle\\
		&\quad +\langle \sigma(t,x,y,z,\kappa,u,\mu,\nu,\lambda),q\rangle, 
\end{aligned}\end{equation}
and denote
\begin{equation*}
	\mathcal{H}^\ast(t):=\mathcal{H}(t,x,y,z,\kappa,p,q,u^\ast,\mu^\ast,\nu^\ast,\lambda^\ast).
\end{equation*}

\begin{mythm}\label{SPM}(Stochastic maximum principle) 
	Let \textbf{(H3.1)} be satisfied. Suppose that  $(x^\ast(\cdot),u^\ast(\cdot))$ is an optimal pair to the \textbf{Problem (P)}, $(\eta(\cdot),\zeta(\cdot),Y(\cdot),Z(\cdot,\cdot))$ is a solution to (\ref{adjoint equation}), $(p(\cdot),q(\cdot))$ are defined by (\ref{defp})-(\ref{defq}). Then, the maximum principle holds:
	\small 
	\begin{equation}\begin{aligned}\label{maximum principle}
			&\bigg<\mathcal{H}^\ast_u(t)+\mathbb{E}^{\mathcal{F}_t}\Big[\mathcal{H}^\ast_{\mu}(t+\delta)\mathbb{I}_{[t_0,T-\delta)}(t)+\int_t^T\phi_2(s,t)^\top \mathcal{H}^\ast_{\nu}(s)ds\\
			&\qquad+\int_t^T\psi_2(s,t)^\top D_t\mathcal{H}^\ast_{\lambda}(s)ds \Big],u-u^{\ast}(t)\bigg>\le 0,\quad\forall u\in U,\ a.e.\ t\in[t_0,T],\ \mathbb{P}\mbox{-}a.s.&
	\end{aligned}\end{equation}
\end{mythm}	
\normalsize 
\begin{proof}	
	By (\ref{variational inequality-2}) and Hamiltonian function (\ref{hamilton}), we obtain
	\begin{equation}\begin{aligned}\label{variationalineq2}
			&\mathbb{E}\int_{t_0}^T\big<\mathcal{H}^\ast_u(t),v(t)\big>dt+\mathbb{E}\int_{t_0}^T\big<\mathcal{H}^\ast_{\mu}(t),v(t-\delta)\big>dt\\
			&+\mathbb{E}\int_{t_0}^T\Big<\mathcal{H}^\ast_{\nu}(t),\int_{t_0}^t\phi_2(t,s)v(s)ds\Big>dt+\mathbb{E}\int_{t_0}^T\Big<\mathcal{H}^\ast_{\lambda}(t),\int_{t_0}^t\psi_2(t,s)v(s)dW(s)\Big>dt\\
			=&\ \mathbb{E}\int_{t_0}^T\big<\mathcal{H}^\ast_u(t),v(t)\big>dt+\mathbb{E}\int_{t_0}^T\big<\mathcal{H}^\ast_{\mu}(t),v(t-\delta)\big>dt+\mathbb{E}\int_{t_0}^T\Big<\int_t^T\phi_2(s,t)^\top \mathcal{H}^\ast_{\nu}(s)ds,v(t) \Big>dt\\
			&+\mathbb{E}\int_{t_0}^T\Big<\int_t^T\psi_2(s,t)^\top \mathbb{E}^{\mathcal{F}_t}\big[D_t\mathcal{H}^\ast_{\lambda}(s)\big]ds,v(t) \Big>dt\le 0.
	\end{aligned}\end{equation}
	Since $u^{\ast}(s)+v(s)$ is an admissible control, we set
	\begin{equation*}
		u^{\ast}(s)+v(s)= 
		\begin{cases}
			u^{\ast}(s), & s\notin[t,t+\varepsilon],\\
			u(s), & s\in[t,t+\varepsilon],
		\end{cases}
	\end{equation*}
	where $t\in[t_0,T]$ and $u(\cdot)$ is an admissible control. Recall (\ref{variationalineq2}), one has
	\begin{equation*}\begin{aligned}
			\frac{1}{\varepsilon}\mathbb{E}\int_t^{t+\varepsilon}\bigg<&\mathcal{H}^\ast_u(s)+\mathbb{E}^{\mathcal{F}_s}\Big[\mathcal{H}^\ast_{\mu}(s+\delta)\mathbb{I}_{[t_0,T-\delta)}(s)+\int_s^T\phi_2(r,s)^\top \mathcal{H}^\ast_{\nu}(r)dr\\
			&+\int_s^T\psi_2(r,s)^\top D_s\mathcal{H}^\ast_{\lambda}(r)dr \Big],u(s)-u^{\ast}(s)\bigg>ds\le 0.
	\end{aligned}\end{equation*}
	Let $\varepsilon\to 0$, which results in
	\begin{equation}\begin{aligned}\label{variineq3}
			\mathbb{E}\bigg<&\mathcal{H}^\ast_u(t)+\mathbb{E}^{\mathcal{F}_t}\Big[\mathcal{H}^\ast_{\mu}(t+\delta)\mathbb{I}_{[t_0,T-\delta)}(t)+\int_t^T\phi_2(s,t)^\top \mathcal{H}^\ast_{\nu}(s)ds\\
			&+\int_t^T\psi_2(s,t)^\top D_t\mathcal{H}^\ast_{\lambda}(s)ds \Big],u(t)-u^{\ast}(t)\bigg>\le 0,\quad a.e.\ t\in[t_0,T].
	\end{aligned}\end{equation}
	For $\forall A\in\mathcal{F}_t$, let $u\in U$ be any deterministic element. Set $w(t)=uI_A+u^{\ast}(t)I_{A^c}$, then $w(\cdot)\in \mathcal{U}_{ad}$. Taking $w(\cdot)$ into (\ref{variineq3}) yields
	\begin{equation*}\begin{aligned}
			\mathbb{E}\bigg[\bigg<&\mathcal{H}^\ast_u(t)+\mathbb{E}^{\mathcal{F}_t}\Big[\mathcal{H}^\ast_{\mu}(t+\delta)\mathbb{I}_{[t_0,T-\delta)}(t)+\int_t^T\phi_2(s,t)^\top \mathcal{H}^\ast_{\nu}(s)ds\\
			&+\int_t^T\psi_2(s,t)^\top D_t\mathcal{H}^\ast_{\lambda}(s)ds \Big],u-u^{\ast}(t)\bigg>\mathbb{I}_A\bigg]\le 0,\quad a.e.\ t\in[t_0,T].
	\end{aligned}\end{equation*}
	Since $A$ is the any element of $\mathcal{F}_t$, which implies
	\begin{equation*}\begin{aligned}\label{variational inequality-3}
			&\mathbb{E}\bigg[\bigg<\mathcal{H}^\ast_u(t)+\mathbb{E}^{\mathcal{F}_t}\Big[\mathcal{H}^\ast_{\mu}(t+\delta)\mathbb{I}_{[t_0,T-\delta)}(t)+\int_t^T\phi_2(s,t)^\top \mathcal{H}^\ast_{\nu}(s)ds\\
			&\qquad +\int_t^T\psi_2(s,t)^\top D_t\mathcal{H}^\ast_{\lambda}(s)ds \Big],u-u^{\ast}(t)\bigg>\bigg|\mathcal{F}_t\bigg]\\
			=&\ \bigg<\mathcal{H}^\ast_u(t)+\mathbb{E}^{\mathcal{F}_t}\Big[\mathcal{H}^\ast_{\mu}(t+\delta)\mathbb{I}_{[t_0,T-\delta)}(t)+\int_t^T\phi_2(s,t)^\top \mathcal{H}^\ast_{\nu}(s)ds\\
			&\qquad +\int_t^T\psi_2(s,t)^\top D_t\mathcal{H}^\ast_{\lambda}(s)ds \Big],u-u^{\ast}(t)\bigg>\le 0,\quad \forall u\in U,\ a.e.\ t\in[t_0,T],\ \mathbb{P}\mbox{-}a.s.
	\end{aligned}\end{equation*}
	We complete this proof.
\end{proof}

\begin{Remark}
	The method is effective. By transforming the variational equation SDDE (\ref{variational equation}) into an SVIE, and introducing an adjoint equation (\ref{adjoint equation}) consisting of a BSDE and a BSVIE, we establish the stochastic maximum principle in Theorem \ref{SPM}. In contrast to the literature \cite{CW10, CZ26, DMOR16, ZF21} and the references therein, the extended mixed delay of the state enters the terminal cost functional rather than the terminal state without delay.
\end{Remark}

\subsection{Verification Theorem}

The stochastic maximum principle has been derived above. Furthermore, under suitable concavity assumptions, the associated maximum condition on the Hamiltonian $\mathcal{H}$ and terminal cost $h$ yield sufficient conditions for the optimality of {\bf Problem (P)}.

Firstly, some assumptions are imposed on $\mathcal{H}$ and $g$ as follows.

(\textbf{H4.1})
The Hamiltonian function $\mathcal{H}(t,x,y,z,\kappa,p,q,u,\mu,\nu,\lambda)$ is concave in $(x,y,z,\kappa,u,\mu,\nu,\lambda)$ and the terminal cost $h(x,y,z,\kappa)$ is concave in $(x,y,z,\kappa)$, for given $p(\cdot),q(\cdot)$.

\begin{mythm}\label{sufficient conition}
	(Verification Theorem) Assume that $\bar{u}(\cdot)\in\mathcal{U}_{ad}$, $\bar{x}(\cdot)$ denotes the corresponding trajectory, $(p(\cdot),q(\cdot))$ defined by (\ref{defp})-(\ref{defq}), where $(\eta(\cdot),\zeta(\cdot),Y(\cdot),Z(\cdot,\cdot))$ is a solution to (\ref{adjoint equation}). Suppose that (\ref{maximum principle}) is satisfied under the assumptions \textbf{(H3.1)} and \textbf{(H4.1)}, then $\bar{u}(\cdot)$ is an optimal control of \textbf{Problem (P)}.   
\end{mythm}

\begin{proof}
	For any admissible control $u(\cdot)\in\mathcal{U}_{ad}$, $x^u(\cdot)$ represents the corresponding trajectory. Let
	\begin{equation*}\begin{aligned}
			\Theta^u(t)&\equiv\big(x^u(t),y^u(t),z^u(t),\kappa^u(t),u(t),\mu(t),\nu(t),\lambda(t)\big),\\
			\bar{\Theta}(t)&\equiv\big(\bar{x}(t),\bar{y}(t),\bar{z}(t),\bar{\kappa}(t),\bar{u}(t),\bar{\mu}(t),\bar{\nu}(t),\bar{\lambda}(t)\big).
	\end{aligned}\end{equation*}
	With the concavity of $\mathcal{H}$ and $g$, we deduce
	\small
	\begin{equation*}\begin{aligned}
			&J\big(u(\cdot)\big)-J\big(\bar{u}(\cdot)\big)\\
			=&\ \mathbb{E}\bigg\{\int_{t_0}^T\Big[\mathcal{H}\big(t,\Theta^u(t),p(t),q(t)\big)-\mathcal{H}\big(t,\bar{\Theta}(t),p(t),q(t)\big)-\big<p(t),b\big(t,\Theta^u(t)\big)-b\big(t,\bar{\Theta}(t)\big)\big>-\big<q(t),\\
			&\qquad \sigma\big(t,\Theta^u(t)\big)-\sigma\big(t,\bar{\Theta}(t)\big)\big>    \Big]dt+h\big(x(T),y(T),z(T),\kappa(T) \big)-h\big(\bar{x}(T),\bar{y}(T),\bar{z}(T),\bar{\kappa}(T)\big)\bigg\}\\
	\end{aligned}\end{equation*}	
	\begin{equation}\begin{aligned}\label{suffi cost func}	
			\le &\ \mathbb{E}\bigg\{\int_{t_0}^T\Big[\big<\mathcal{H}_u\big(t,\bar{\Theta}(t),p(t),q(t)\big),u(t)-\bar{u}(t)\big>+\big<\mathcal{H}_{\mu}\big(t,\bar{\Theta}(t),p(t),q(t)\big),\mu(t)-\bar{\mu}(t)\big>\\
			&\quad+\big<\mathcal{H}_{\nu}\big(t,\bar{\Theta}(t),p(t),q(t)\big),\nu(t)-\bar{\nu}(t)\big>+\big<\mathcal{H}_{\lambda}\big(t,\bar{\Theta}(t),p(t),q(t)\big),\lambda(t)-\bar{\lambda}(t)\big>\\
			&\quad+\big<\mathcal{H}_x\big(t,\bar{\Theta}(t),p(t),q(t)\big),x(t)-\bar{x}(t)\big>+\big<\mathcal{H}_y\big(t,\bar{\Theta}(t),p(t),q(t)\big),y(t)-\bar{y}(t)\big>\\
			&\quad+\big<\mathcal{H}_z\big(t,\bar{\Theta}(t),p(t),q(t)\big),z(t)-\bar{z}(t)\big>+\big<\mathcal{H}_{\kappa}\big(t,\bar{\Theta}(t),p(t),q(t)\big),\kappa(t)-\bar{\kappa}(t)\big>\\
			&\quad-\big<p(t),b\big(t,\Theta^u(t)\big)-b\big(t,\bar{\Theta}(t)\big)\big>-\big<q(t),\sigma\big(t,\Theta^u(t)\big)-\sigma\big(t,\bar{\Theta}(t)\big)\big>\Big]dt+\mathbb{H}\mathring{X}(T)\bigg\},
	\end{aligned}\end{equation}
	\normalsize
	where 
	\begin{equation*}
		\mathring{X}(T)^\top:=\begin{bmatrix}
			x(t)^\top-\bar{x}(t)^\top&
			y(t)^\top-\bar{y}(t)^\top&
			z(t)^\top-\bar{z}(t)^\top&
			\kappa(t)^\top-\bar{\kappa}(t)^\top
		\end{bmatrix}.
	\end{equation*}
	From (\ref{state equation}) and the derivation process of (\ref{FSVIE}), the following results are gained:
	\small
	\begin{equation}\label{suffstateeq}
		\mathring{X}(t)=\int_{t_0}^T\bar{\mathbb{A}}(t,s)\big[b\big(s,\Theta^u(s)\big)-b\big(s,\bar{\Theta}(s)\big)\big]ds+\int_{t_0}^T\bar{\mathbb{A}}(t,s)\big[\sigma\big(s,\Theta^u(s)\big)-\sigma\big(s,\bar{\Theta}(s)\big)\big]dW(s),
	\end{equation}
	\normalsize
	where 
	\begin{equation*}
		\bar{\mathbb{A}}(t,s):=\begin{bmatrix}
			I\\
			\mathbb{I}_{(\delta,\infty)}(t-s)\\
			\mathcal{E}_1(t,s)\\
			\mathcal{E}_2(t,s)
		\end{bmatrix}
		=\begin{bmatrix}
			I\\
			\mathbb{I}_{(\delta,\infty)}(t-s)\\
			\mathcal{E}_1(t,s)\\
			0
		\end{bmatrix}+
		\begin{bmatrix}
			0\\	0\\	0\\
			\mathcal{E}_2(t,s)
		\end{bmatrix}
		:=\bar{\mathbb{A}}_1(t,s)+\bar{\mathbb{A}}_2(t,s).
	\end{equation*}
	
	Similar to the treating of (\ref{deal Terminal term}), by combining with (\ref{adjoint equation}\ a), we derive
	\small
	\begin{equation}\begin{aligned}\label{Terminal term BSDE}
			&\mathbb{E}\big[\mathbb{H}\mathring{X}(T)\big]\\
			=&\ \mathbb{E}\int_{t_0}^T\bigg[\Big<\eta^1(t)+\eta^2(t)\mathbb{I}_{[t_0,T-\delta)}(t)+\mathcal{E}_1(T,t)^\top\eta^3(t)+\int_t^T\psi_1(T,s)^\top \mathbb{E}^{\mathcal{F}_s}\big[D_sh_{\kappa}(T)^\top\big] ds,b\big(t,\Theta^u(t)\big)\\
			&\quad-b\big(t,\bar{\Theta}(t)\big)\Big>+\Big<\zeta^1(t)+\zeta^2(t)\mathbb{I}_{[t_0,T-\delta)}(t)+\mathcal{E}_1(T,t)^\top\zeta^3(t)+\int_t^T\mathbb{E}^{\mathcal{F}_t}\big[D_t\big(\psi_1(T,s)^\top\zeta^4(s)\big)\big]ds,\\
			&\qquad\sigma\big(t,\Theta^u(t)\big)-\sigma\big(t,\bar{\Theta}(t)\big)\Big>\bigg]dt+\mathbb{E}\int_{t_0}^T\big<\mathring{X}(t),\psi(t)\big>dt-\mathbb{E}\int_{t_0}^T\mathbb{L}(t)\mathring{X}(t)dt.
	\end{aligned}\end{equation}
	\normalsize
	Recall back (\ref{adjoint equation}\ b), by a series of calculations, we obtain
	\small
	\begin{equation}\begin{aligned}\label{Terminal term BSVIE}
			&\mathbb{E}\int_{t_0}^T\big<\mathring{X}(t),\psi(t)\big>dt\\
			=\ &\mathbb{E}\int_{t_0}^T\bigg[\Big<\int_t^T Y^1(s)ds+\int_{t+\delta}^TY^2(s)ds\mathbb{I}_{[t_0,T-\delta)}(t)+\int_t^T\mathcal{E}_1(s,t)^\top Y^3(s)ds+\int_t^T\int_t^s\psi_1(s,r)^\top\\
			&\quad\times\mathbb{E}^{\mathcal{F}_r}\big[D_rY^4(s)\big]drds,b\big(t,\Theta^u(t)\big)-b\big(t,\bar{\Theta}(t)\big)\Big>+\Big<\int_t^T Z^1(s,t)ds+\int_{t+\delta}^T Z^2(s,t)ds\mathbb{I}_{[t_0,T-\delta)}(t)\\
			&\quad +\int_t^T \mathcal{E}_1(s,t)^\top Z^3(s,t)ds+\int_t^T\int_t^s\mathbb{E}^{\mathcal{F}_t}\big[D_t\big(\psi_1(s,r)^\top Z^4(s,r)\big)\big]drds,\sigma\big(t,\Theta^u(t)\big)-\sigma\big(t,\bar{\Theta}(t)\big)\Big>\bigg]dt\\
			&-\mathbb{E}\int_{t_0}^T\Big<\mathring{X}(t),\int_t^T\bigg\{\mathbb{A}_1(s,t)^\top Y(s)+\mathbb{C}_1(s,t)^\top Z(s,t)+\int_t^s\Big\{\hat{\mathbb{A}}_2(t)^\top \psi_1(s,r)^\top \mathbb{E}^{\mathcal{F}_r}\big[D_rY(s)\big]\\
			&\quad\quad+\hat{\mathbb{C}}_2(t)^\top\mathbb{E}^{\mathcal{F}_t}\Big[D_t\Big(\psi_1(s,r)^\top Z(s,r)\Big)\Big]\Big\}dr\bigg\}ds\Big>dt-\mathbb{E}\int_{t_0}^T\Big<\mathring{X}(t),\mathbb{L}(t)^\top \Big>dt.
	\end{aligned}\end{equation} 
	\normalsize 
	Taking (\ref{Terminal term BSDE}) and (\ref{Terminal term BSVIE}) into (\ref{suffi cost func}) leads to 
	\small
	\begin{equation}\begin{aligned}\label{suffcofunc3}
			&J\big(u(\cdot)\big)-J\big(\bar{u}(\cdot)\big)\\
			\le &\ \mathbb{E}\bigg\{\int_{t_0}^T\Big[\big<\mathcal{H}_u\big(t,\bar{\Theta}(t),p(t),q(t)\big),u(t)-\bar{u}(t)\big>+\big<\mathcal{H}_{\mu}\big(t,\bar{\Theta}(t),p(t),q(t)\big),\mu(t)-\bar{\mu}(t)\big>\\
			&\quad+\big<\mathcal{H}_{\nu}\big(t,\bar{\Theta}(t),p(t),q(t)\big),\nu(t)-\bar{\nu}(t)\big>+\big<\mathcal{H}_{\lambda}\big(t,\bar{\Theta}(t),p(t),q(t)\big),\lambda(t)-\bar{\lambda}(t)\big>\\
			&\quad+\big<\mathcal{H}_x\big(t,\bar{\Theta}(t),p(t),q(t)\big),x(t)-\bar{x}(t)\big>+\big<\mathcal{H}_y\big(t,\bar{\Theta}(t),p(t),q(t)\big),y(t)-\bar{y}(t)\big>\\
			&\quad+\big<\mathcal{H}_z\big(t,\bar{\Theta}(t),p(t),q(t)\big),z(t)-\bar{z}(t)\big>+\big<\mathcal{H}_{\kappa}\big(t,\bar{\Theta}(t),p(t),q(t)\big),\kappa(t)-\bar{\kappa}(t)\big>\Big]dt\bigg\}\\
			&-\mathbb{E}\int_{t_0}^T\Big<\mathring{X}(t),\int_t^T\bigg\{\mathbb{A}_1(s,t)^\top Y(s)+\mathbb{C}_1(s,t)^\top Z(s,t)+\int_t^s\Big\{\hat{\mathbb{A}}_2(t)^\top \psi_1(s,r)^\top \mathbb{E}^{\mathcal{F}_r}\big[D_rY(s)\big]\\
			&\quad\quad+\hat{\mathbb{C}}_2(t)^\top\mathbb{E}^{\mathcal{F}_t}\big[D_t\big(\psi_1(s,r)^\top Z(s,r)\big)\big]\Big\}dr\bigg\}ds\Big>dt-\mathbb{E}\int_{t_0}^T\big<\mathring{X}(t),\mathbb{L}(t)^\top\big>dt.	
	\end{aligned}\end{equation}
	\normalsize 
	Observe (\ref{suffstateeq}) and combine with (\ref{varphiY}) as well as (\ref{dual principle}), which results in 
	\begin{equation}\label{realtionship}
		\mathbb{E}\int_{t_0}^T\big<\mathring{X}(t),\psi(t)\big>dt-\mathbb{E}\int_{t_0}^T\big<\mathring{X}(t),\mathbb{L}(t)^\top\big>dt=0.
	\end{equation}
	Taking (\ref{adjointeqvector}) and (\ref{realtionship}) into (\ref{suffcofunc3}), one has
	\small
	\begin{equation*}\begin{aligned}
			&\mathbb{E}\int_{t_0}^T\Big<\mathring{X}(t),\int_t^T\bigg\{\mathbb{A}_1(s,t)^\top Y(s)+\mathbb{C}_1(s,t)^\top Z(s,t)+\int_t^s\Big\{\hat{\mathbb{A}}_2(t)^\top \psi_1(s,r)^\top \mathbb{E}^{\mathcal{F}_r}\big[D_rY(s)\big]\\
			&\quad\quad+\hat{\mathbb{C}}_2(t)^\top\mathbb{E}^{\mathcal{F}_t}\big[D_t\big(\psi_1(s,r)^\top Z(s,r)\big)\big]\Big\}dr\bigg\}ds\Big>dt+\mathbb{E}\int_{t_0}^T\big<\mathring{X}(t),\psi(t)\big>dt\\
			=&\ \mathbb{E}\int_{t_0}^T\Big[ \big<b_x(t)p(t)+\sigma_x(t)q(t)+l_x(t),x(t)-\bar{x}(t)\big>+\big<b_y(t)p(t)+\sigma_y(t)q(t)+l_y(t),y(t)-\bar{y}(t)\big>\\
			&\quad+\big<b_z(t)p(t)+\sigma_z(t)q(t)+l_z(t),z(t)-\bar{z}(t)\big>+\big<b_{\kappa}(t)p(t)+\sigma_{\kappa}(t)q(t)+l_{\kappa}(t),\kappa(t)-\bar{\kappa}(t)\big>\Big]dt.
	\end{aligned}\end{equation*}
	\normalsize 
	Then, by the maximum condition (\ref{maximum principle}), we obtain
	\begin{equation*}
		J\big(u(\cdot)\big)-J\big(\bar{u}(\cdot)\big)\le 0.
	\end{equation*}
	This proof is completed.
\end{proof}

\begin{Remark}
	Since the appearance of $u(t-\delta),\ \int_{t_0}^t\phi_2(t,s)u(s)ds,\ \int_{t_0}^t\psi_2(t,s)u(s)dW(s)$ in \textbf{Problem (P)}, the maximum principle (\ref{maximum principle}) contains four parts: $\mathbb{E}^{\mathcal{F}_t}\big[\int_t^T\psi_2(s,t)^\top D_t\mathcal{H}^\ast_{\lambda}(s)ds\big]$ characterizes the maximum condition with $\int_{t_0}^t\psi_2(t,s)u(s)dW(s)$, $\mathbb{E}^{\mathcal{F}_t}\big[\int_t^T\phi_2(s,t)^\top\mathcal{H}^\ast_{\nu}(s)ds\big]$ characterizes the maximum condition with $\int_{t_0}^t\phi_2(t,s)u(s)ds$, $\mathbb{E}^{\mathcal{F}_t}\big[\mathcal{H}^\ast_{\mu}(t+\delta)\mathbb{I}_{[t_0,T )}(t)\big]$ characterizes the maximum condition with $u(t-\delta)$, while $\mathcal{H}^\ast_{u}(t)$ characterizes the one without delay which naturally reduces to the SDEs case. Furthermore, when $\phi_2(\cdot,\cdot)\equiv\psi_2(\cdot,\cdot)\equiv0$, the maximum principle condition (\ref{maximum principle}) reduces (13) in \cite{CW10}. In addition, \cite{ZF21} deal with the system containing the term $\int_{-\delta}^0e^{\lambda s}u(s+\delta)ds$, but the presented maximum condition (H1) is similar to the first of (\ref{variationalineq2}) rather than the expression of (\ref{maximum principle}).
	
\end{Remark}

\section{Transform BSVIE into ABSDE}

To facilitate the subsequent analysis, recall back (\ref{defp})-(\ref{defq}) as follows:
\begin{equation}\left\{\begin{aligned}\label{defpq2}
		p(t):&=\eta^1(t)+\eta^2(t)\mathbb{I}_{[t_0,T-\delta)}(t)+\mathcal{E}_1(T,t)^\top\eta^3(t)+\mathbb{E}^{\mathcal{F}_t}\Big[\int_t^T\psi_1(T,r)^\top D_r\eta^4(r)dr\Big]\\
		&\quad+\mathbb{E}^{\mathcal{F}_t}\Big[\int_t^T Y^1(s)ds+\int_{t+\delta}^T Y^2(s)ds\mathbb{I}_{[t_0,T-\delta)}(t)\\
		&\quad+\int_t^T \mathcal{E}_1(s,t)^\top Y^3(s)ds+\int_t^T\int_t^s\psi_1(s,r)^\top D_rY^4(s)drds\Big ],\\
		q(t):&=\zeta^1(t)+\zeta^2(t)\mathbb{I}_{[t_0,T-\delta)}(t)+\mathcal{E}_1(T,t)^\top\zeta^3(t)+\int_t^T\mathbb{E}^{\mathcal{F}_t}\big[D_t\big(\psi_1(T,r)^\top\zeta^4(r)\big)\big]dr\\
		&\quad+\int_t^T Z^1(s,t)ds+\int_{t+\delta}^T Z^2(s,t)ds\mathbb{I}_{[t_0,T-\delta)}(t)\\
		&\quad+\int_t^T \mathcal{E}_1(s,t)^\top Z^3(s,t)ds+\int_t^T\int_t^s\mathbb{E}^{\mathcal{F}_t}\big[D_t\big(\psi_1(s,r)^\top Z^4(s,r)\big)\big]drds.\\
	\end{aligned}\right.\end{equation}

We now construct a link between the adjoint equation (\ref{adjoint equation}) and a class of ABSDEs. 

\begin{mythm}\label{the5.1}
	Suppose that \textbf{(H3.1)} be satisfied. Let $(x^{\ast}(\cdot),u^{\ast}(\cdot))$ be an optimal pair, $(\eta(\cdot),\zeta(\cdot),\\Y(\cdot),Z(\cdot,\cdot))$ is the solution of (\ref{adjoint equation}. Assume further that $\mathcal{E}_1(\cdot,\cdot)\equiv C$ for some constant $C$ and $\psi_1(\cdot,\cdot)$ is a deterministic function. Then, the pair $(p(\cdot),q(\cdot))$ defined by (\ref{defpq2}) satisfies the following system of ABSDEs:    
	\begin{equation}\left\{\begin{aligned}\label{ABSDE}
			p(t)=&\ h_x(T)^\top+\mathcal{E}_1(T,t)^\top h_z(T)^\top+\int_t^T\psi_1(T,s)^\top\mathbb{E}^{\mathcal{F}_s}\big[D_sh_{\kappa}(T)^\top\big]ds\\
			&+\int_t^T\Big\{l_x(s)+b_x(s)^\top p(s)+\sigma_x(s)^\top q(s)+\mathcal{E}_1(s,t)^\top\big[l_z(s)+b_z(s)^\top p(s)\\
			&+\sigma_z(s)^\top q(s) \big]+\int_t^s\psi_1(s,r)^\top D_r\big[l_{\kappa}(s)+b_{\kappa}(s)^\top p(s)+\sigma_{\kappa}(s)^\top q(s)\big]dr  \Big\}ds\\
			&-\int_t^Tq(s)dW(s),\quad  t\in[T-\delta,T],\\         
			p(t)=&\ p(T-\delta)+\mathbb{E}^{\mathcal{F}_{T-\delta}}\big[h_y(T)\big]+\int_t^T\Big\{l_x(s)+b_x(s)^\top p(s)+\sigma_x(s)^\top q(s)\\
			&+\mathcal{E}_1(s,t)^\top\big[l_z(s)+b_z(s)^\top p(s)+\sigma_z(s)^\top q(s) \big]+\int_t^s\psi_1(s,r)^\top D_r\big[l_{\kappa}(s)\\
			&+b_{\kappa}(s)^\top p(s)+\sigma_{\kappa}(s)^\top q(s)\big]dr+\mathbb{E}^{\mathcal{F}_s}\big[l_y(s+\delta)+b_y(s+\delta)^\top p(s+\delta)\\
			&+\sigma_y(s+\delta)^\top q(s+\delta) \big]  \Big\}ds-\int_t^Tq(s)dW(s),\quad t\in[t_0,T-\delta).
		\end{aligned}\right.\end{equation}
\end{mythm}

\begin{proof}
	Observing equation (\ref{ABSDE}), it can be written in a unified form as follows:
	\begin{equation}\begin{aligned}
			p(t)=&\ p(T-\delta)+\mathbb{E}^{\mathcal{F}_{T-\delta}}\big[h_y(T)\big]+\mathcal{E}_1(T,t)^\top h_z(T)^\top+\int_t^T\psi_1(T,s)^\top\mathbb{E}^{\mathcal{F}_s}\big[D_sh_{\kappa}(T)^\top\big]ds\\
			&+\int_t^T\Big\{l_x(s)+b_x(s)^\top p(s)+\sigma_x(s)^\top q(s)+\mathcal{E}_1(s,t)^\top\big[l_z(s)+b_z(s)^\top p(s)+\sigma_z(s)^\top\\
			&\times q(s) \big]+\int_t^s\psi_1(s,r)^\top D_r\big[l_{\kappa}(s)+b_{\kappa}(s)^\top p(s)+\sigma_{\kappa}(s)^\top q(s)\big]dr+\mathbb{E}^{\mathcal{F}_s}\big[l_y(s+\delta)\\
			&+b_y(s+\delta)^\top p(s+\delta)+\sigma_y(s+\delta)^\top q(s+\delta) \big]\Big\}ds-\int_t^Tq(s)dW(s).
	\end{aligned}\end{equation}
	Without loss of generality, let $\mathcal{E}_1(\cdot,\cdot)\equiv 1$ for simplicity. Applying conditional expectations to both sides of (\ref{adjoint equation}), noting (\ref{adjointeqvector}), we derive the following results:
	\begin{equation*}\begin{aligned}
			&\mathbb{E}^{\mathcal{F}_t}\big[Y^1(t)+Y^2(t+\delta)\mathbb{I}_{[t_0,T-\delta)}(t)+Y^3(t) \big]\\
			=&\ b_x(t)^\top p(t)+\sigma_x(t)^\top q(t)+l_x(t)+b_z(t)^\top p(t)+\sigma_z(t)^\top q(t)+l_z(t)\\
			&+\mathbb{E}^{\mathcal{F}_t}\big[b_y(t+\delta)^\top p(t+\delta)+\sigma_y(t+\delta)^\top q(t+\delta)+l_y(t+\delta)\big]\mathbb{I}_{[t_0,T-\delta)}(t),
	\end{aligned}\end{equation*}
	and 
	\begin{equation*}
		Y^4(t)=b_{\kappa}(t)^\top p(t)+\sigma_{\kappa}(t)^\top q(t)+l_{\kappa}(t),\quad t_0\le t \le T.
	\end{equation*}
	Since
	\begin{equation*}\begin{aligned}
			&\mathbb{E}^{\mathcal{F}_s}\Big[ \int_t^T\int_t^{s+\delta}Z^2(s+\delta,r)dW(r)\mathbb{I}_{[t_0,T-\delta)}(s)ds \Big]\\
			=&\ \mathbb{E}^{\mathcal{F}_s}\Big[ \int_t^T\int_t^sZ^2(s+\delta,r)dW(r)\mathbb{I}_{[t_0,T-\delta)}(s)ds+\int_t^T\int_s^{s+\delta}Z^2(s+\delta,r)dW(r)\mathbb{I}_{[t_0,T-\delta)}(s)ds\Big]\\
			=&\int_t^T\mathbb{E}^{\mathcal{F}_s}\Big[\int_t^sZ^2(s+\delta,r)dW(r)\mathbb{I}_{[t_0,T-\delta)}(s)\Big]ds\\
			=&\int_t^{T-\delta}\int_t^sZ^2(s+\delta,r)dW(r)ds\mathbb{I}_{[t_0,T-\delta)}(t).
	\end{aligned}\end{equation*}
	Then, by stochastic Fubini's theorem, one has
	\begin{equation*}\begin{aligned}
			&\int_t^T\mathbb{E}^{\mathcal{F}_s}\big[Y^1(s)+Y^2(s+\delta)\mathbb{I}_{[t_0,T-\delta)}(s)+Y^3(s)\big]ds\\
			=&\int_t^T\mathbb{E}^{\mathcal{F}_s}\Big[\mathbb{E}^{\mathcal{F}_t}\big[Y^1(s)\big]+\int_t^sZ^1(s,r)dW(r)+\mathbb{E}^{\mathcal{F}_t}\big[Y^2(s+\delta)\mathbb{I}_{[t_0,T-\delta)}(s)\big]\\
			&\qquad\qquad+\int_t^{s+\delta}Z^2(s,r)\mathbb{I}_{[0,T-\delta)}(s)dW(r)+\mathbb{E}^{\mathcal{F}_t}\big[Y^3(s)\big]+\int_t^sZ^3(s,r)dW(r)\Big]ds\\
			=&\int_t^T\mathbb{E}^{\mathcal{F}_t}\Big[Y^1(s)+Y^2(s+\delta)\mathbb{I}_{[t_0,T-\delta)}(s)+Y^3(s) \Big]ds\\
			&+\int_t^T\Big[\int_s^TZ^1(r,s)dr+\int_s^{T-\delta}Z^2(r+\delta,s)dr\mathbb{I}_{[t_0,T-\delta)}(s)\mathbb{I}_{[t_0,T-\delta)}(t)+\int_s^TZ^3(r,s)dr\Big]dW(s).
	\end{aligned}\end{equation*}
	Next, by Clark-Ocone's formula, the following result holds:
	\begin{equation*}\begin{aligned}
			&\int_t^T\mathbb{E}^{\mathcal{F}_s}\bigg[\int_t^s\mathbb{E}^{\mathcal{F}_r}\big[D_rY^4(s) \big]dr\bigg]ds\\
			=&\int_t^T\mathbb{E}^{\mathcal{F}_s}\bigg[\int_t^s\Big(\mathbb{E}^{\mathcal{F}_t}\big[D_rY^4(s) \big]+\int_t^r\mathbb{E}^{\mathcal{F}_{\theta}}\big[D_{\theta}Z^4(s,r)\big]dW(\theta)\Big)dr\bigg]ds\\
			=&\int_t^T\int_t^s\mathbb{E}^{\mathcal{F}_t}\big[D_rY^4(s)\big]drds+\int_t^T\int_t^s\int_r^s\mathbb{E}^{\mathcal{F}_r}\big[D_rZ^4(s,\theta)\big]d\theta dW(r)ds\\
			=&\int_t^T\int_t^s\mathbb{E}^{\mathcal{F}_t}\big[D_rY^4(s)\big]drds+\int_t^T\int_s^T\int_s^r\mathbb{E}^{\mathcal{F}_s}\big[D_sZ^4(r,\theta)\big]d\theta drdW(s),
	\end{aligned}\end{equation*}
	Recalling (\ref{adjoint equation}) with Proposition \ref{pro2.5}, we gain 
	\begin{equation*}
		\mathbb{E}^{\mathcal{F}_s}\big[D_sh_{\kappa}(T)^\top\big]=\zeta^4(s)=\mathbb{E}^{\mathcal{F}_t}\big[\zeta^4(s)\big]+\int_t^s\mathbb{E}^{\mathcal{F}_r}\big[D_r\zeta^4(s)\big]dW(r).
	\end{equation*}
	Then 
	\small
	\begin{equation*}\begin{aligned}
			&\eta^1(t)+\eta^2(t)\mathbb{I}_{[t_0,T-\delta)}(t)+\eta^3(t)+\int_t^T\psi_1(T,r)^\top \mathbb{E}^{\mathcal{F}_t}\big[D_s\eta^4(s)\big]ds+\int_t^T \mathbb{E}^{\mathcal{F}_t}\Big[Y^1(s)\\
			&+Y^2(s)\mathbb{I}_{[t_0,T-\delta)}(t)\mathbb{I}_{[t_0,T-\delta)}(s)+Y^3(s)+\int_t^s\psi_1(s,r)^\top D_rY^4(s)dr\Big]ds+\int_t^T\bigg[\zeta^1(s)\\
			&+\zeta^2(s)\mathbb{I}_{[t_0,T-\delta)}(s) +\zeta^3(s)+\int_s^T\psi_1(T,r)^\top\mathbb{E}^{\mathcal{F}_s}\big[D_s\zeta^4(r)\big]dr+\int_s^T \Big(Z^1(r,s)\\
			&+Z^2(r+\delta,s)\mathbb{I}_{[t_0,T-\delta)}(r)\mathbb{I}_{[t_0,T-\delta)}(t)+Z^3(r,s)+\int_s^r\psi_1(r,\theta)^\top\mathbb{E}^{\mathcal{F}_s}\big[D_s Z^4(r,\theta)\big]d\theta\Big) dr\bigg]dW(s)\\
			=&\ h(T)^\top+\mathbb{E}^{\mathcal{F}_{T-\delta}}\big[h_y(T)^\top\big]\mathbb{I}_{[t_,T-\delta)}(t)+h_z(T)^\top+\int_t^T\psi_1(T,s)^\top\mathbb{E}^{\mathcal{F}_s}\big[D_sh_{\kappa}(T)^\top\big]ds+\int_t^T\Big\{l_x(s)\\
			&+b_x(s)^\top p(s)+\sigma_x(s)^\top q(s)+l_z(s)+b_z(s)^\top p(s)+\sigma_z(s)^\top q(s)+\int_t^s\psi_1(s,r)^\top D_r\big[l_{\kappa}(s) \\
			&+b_{\kappa}(s)^\top p(s)+\sigma_{\kappa}(s)^\top q(s)\big]dr+\mathbb{E}^{\mathcal{F}_s}\big[l_y(s+\delta)
			+b_y(s+\delta)^\top p(s+\delta)
			+\sigma_y(s+\delta)^\top q(s+\delta) \big]\Big\}ds.		
	\end{aligned}\end{equation*}
	\normalsize
	Thus, the proof is completed.
\end{proof}

\begin{Remark}
	We explain here why we assume that $\mathcal{E}_1(\cdot,\cdot)\equiv C$ and $\psi_1(\cdot,\cdot)$ is a deterministic function in Theorem \ref{the5.1}. In fact, when we try to convert the adjoint equation (\ref{adjoint equation}) from a BSVIE with Malliavin derivatives, into a ABSDE (\ref{ABSDE}), we apply the martingale representation theorem to $Y^3(\cdot)$ to obtain $\int_t^T\mathbb{E}^{\mathcal{F}_s}[\mathcal{E}_1(s,t)^\top Y^3(s)]ds=\int_t^T\mathbb{E}^{\mathcal{F}_t}[\mathcal{E}_1(s,t)^\top Y^3(s)]ds+\int_t^T\int_s^T\mathcal{E}_1(r,t)^\top Z^3(r,s)drdW(s)$, but from the definition (\ref{defq}) of $q(\cdot)$, we needs the term like $\int_t^T\int_s^T\mathcal{E}_1(r,s)^\top Z^3(r,s)drdW(s)$ appears. So we have to assume that $\mathcal{E}_1(\cdot,\cdot)\equiv C$.
In the meanwhile, by Clark-Ocone's formula, $\int_t^T\mathbb{E}^{\mathcal{F}_s}[\int_t^s\psi_1(s,r)^\top D_rY^3(s)dr]ds=\int_t^T\int_t^s\psi_1(s,r)^\top \mathbb{E}^{\mathcal{F}_t}[Y^3(s)]\\drds+\int_t^T\int_t^s\int_t^s\psi_1(s,r)^\top D_rZ(s,\theta)dW(\theta)drds$, but from the definition (\ref{defp}) of $p(\cdot)$, we need the term $\mathbb{E}^{\mathcal{F}_t}[\int_t^T\int_t^s\psi_1(s,r)^\top D_rY^3(s)drds]$ arises, so $\psi_1(\cdot,\cdot)$ is supposed to be a deterministic function.
We wish to relax these assumptions in the future.
\end{Remark}

\begin{Remark}
	(\ref{ABSDE}) establishes a connection between the adjoint equation (\ref{adjoint equation}), which consists of a BSDE with a BSVIE, and an ABSDE, which allows for a more concise representation of (\ref{adjoint equation}). This is different from \cite{CZ26, ZF21}, which needs to introduce multiple ABSDEs as adjoint equations to give the maximum principle or verification theorem.
\end{Remark}

\section{Applications}

\subsection{Nonzero-sum stochastic differential game with extended mixed delays}

As discussed in Section 1, the strategic interactions among heterogeneous investors in the stock market can be modeled as a nonzero-sum game. This motivates us the investigation of nonzero-sum stochastic differential games with extended mixed delays in the control variable and state variable, which is of both practical importance and theoretical significance. In this section, we restrict our attention to the two-player case, the generalization to $n\ge 3$ players follows analogously.

The controlled stochastic system is 
\begin{equation}\hspace{-4mm}\left\{\begin{aligned}\label{state equation NZG}
		dx(t)&=\hat{b}\big(t,x(t),y(t),z(t),\kappa(t),u_1(t),\mu_1(t),\nu_1(t),\\
		&\qquad\quad\lambda_1(t),u_2(t),\mu_2(t),\nu_2(t), \lambda_2(t)\big)dt\\
		&\quad+\hat{\sigma}\big(t,x(t),y(t),z(t),\kappa(t),u_1(t),\mu_1(t),\nu_1(t),\\
		&\qquad\quad\lambda_1(t),u_2(t),\mu_2(t),\nu_2(t), \lambda_2(t)\big)dW(t),\quad t\in[t_0,T],\\
		x(t)&=\hat{\xi}(t),\ u_1(t)=\hat{\varsigma}_1(t),\ u_2(t)=\hat{\varsigma}_2(t),\quad t\in[t_0-\delta,t_0],
	\end{aligned}\right.\end{equation}
with the cost functional
\begin{equation}\begin{aligned}\label{cost NZG}
		J_i\big(u_1(\cdot),u_2(\cdot)\big)=\mathbb{E}\bigg\{\int_{t_0}^{T}l_i\big(&t,x(t),y(t),z(t),\kappa(t),u_1(t),\mu_1(t),\nu_1(t),\lambda_1(t),u_2(t),\mu_2(t),\\
		&\nu_2(t), \lambda_2(t)\big)dt+h_i\big(x(T),y(T),z(T),\kappa(T)\big)\bigg\},\ i=1,2,
\end{aligned}\end{equation}
where $u_1(\cdot)$ is the control process of Player 1 with $u_2(\cdot)$ signifying the control process of Player 2. $\mu_i(\cdot),\nu_i(\cdot),\lambda_i(\cdot)$ are defined similar to $\mu(\cdot),\nu(\cdot),\lambda(\cdot)$ in Section 1 respectively, and just need to replace $u(\cdot), \phi_2(\cdot,\cdot), \psi_2(\cdot,\cdot)$ by $u_i(\cdot), \phi_{2i}(\cdot,\cdot), \psi_{2i}(\cdot,\cdot)$, respectively, $i=1,2$.  Assume that $U_1,U_2\in\mathbb{R}^m$ are two nonempty convex set. Let
\begin{equation}\begin{aligned}\label{admissible control set nonzero}
		\mathcal{U}_i[t_0,T]:=\Big\{u_i(\cdot):[t_0-\delta,T]\to\mathbb{R}^m \big|& u_i(\cdot)~ \text{is a}~
		U_i\text{-valued, square-integrable},\ \mathbb{F}\text{-adapted}\\
		&\text{process and}~u_i(t)=\hat{\varsigma}_i(t),~t\in[t_0-\delta,t_0]\Big\}  	
\end{aligned}\end{equation}
be the admissible control set of Player $i,\ i=1,2$.
Then the nonzero-sum stochastic differential game with extended mixed delays is studied in this section as follows.

\textbf{Problem (NZG).}\ To find a pair of control $(u_1^\ast(\cdot),u_2^\ast(\cdot))\in\mathcal{U}_1[t_0,T]\times\mathcal{U}_2[t_0,T]$ such that (\ref{state equation NZG}) is satisfied and 
\begin{equation}\left\{\begin{aligned}\label{nash equilibrium}     
		J_1\big(u_1(\cdot),u_2^\ast(\cdot) \big)&\le J_1\big(u_1^\ast(\cdot),u_2^\ast(\cdot) \big),\quad \forall u_1(\cdot)\in\mathcal{U}_1[t_0,T],\\ 
		J_2\big(u_1^\ast(\cdot),u_2(\cdot) \big)&\le J_2\big(u_1^\ast(\cdot),u_2^\ast(\cdot) \big),\quad \forall u_2(\cdot)\in\mathcal{U}_2[t_0,T]. 
	\end{aligned}\right.\end{equation}
Any $(u_1^\ast(\cdot),u_2^\ast(\cdot))$ is called a \emph{Nash equilibrium point} of \textbf{Problem (NZG)}, if (\ref{nash equilibrium}) is satisfied.

The following assumptions are imposed on (\ref{state equation NZG}) and (\ref{cost NZG}).

(\textbf{H6.1})
(i) $\hat{b},\hat{\sigma},l_i,h_i$ are continuous differentiable in $(x,y,z,\kappa,u_1,\mu_1,\nu_1,\lambda_1,u_2,\mu_2,\nu_2,\lambda_2)$. They with their derivatives $f_\varrho$ are continuous and uniformly bounded in $(x,y,z,\kappa,u_1,\mu_1,\nu_1,\lambda_1,\\u_2,\mu_2,\nu_2,\lambda_2)$, where $\varrho=x,y,z,\kappa,u_1,\mu_1,\nu_1,\lambda_1,u_2,\mu_2,\nu_2,\lambda_2$, $f=\hat{b},\hat{\sigma},l_i,h_i,\ i=1,2$.

(ii) $|\hat{b}(t,0,0,0,0,u_1,\mu_1,\nu_1,\lambda_1,u_2,\mu_2,\nu_2,\lambda_2)|+|\hat{\sigma}(t,0,0,0,0,u_1,\mu_1,\nu_1,\lambda_1,u_2,\mu_2,\nu_2,\lambda_2)|\\+|l_i(t,0,0,0,0,u_1,\mu_1,\nu_1,\lambda_1,u_2,\mu_2,\nu_2,\lambda_2)|
+|h_i(0,0,0,0)|+|l_{i\varrho}(t,0,0,0,0,u_1,\mu_1,\nu_1,\lambda_1,u_2,\mu_2,\\\nu_2,\lambda_2)|+|h_{i\varrho}(0,0,0,0)|\le C$ for some $C$, where $\varrho=x,y,z,\kappa,u_1,\mu_1,\nu_1,\lambda_1,u_2,\mu_2,\nu_2,\lambda_2,\ i=1,2$.

(iii) The initial trajectory of the state $\hat{\xi}(\cdot)$ is a deterministic continuous function, and the initial trajectory of the control $\hat{\varsigma}_i(\cdot)$ is a deterministic square integrable function. The autoregressive kernel $\phi_{1}(\cdot,\cdot)$ and the moving average kernel $\psi_{1}(\cdot,\cdot)$ belong to $L^{\infty}\big([t_0,T];L_{\mathbb{F}}^{\infty}([t_0,T];\mathbb{R}^{n\times n})\big),$ and $\phi_{2i}(\cdot,\cdot),\psi_{2i}(\cdot,\cdot)\in L^{\infty}\big([t_0,T];L_{\mathbb{F}}^{\infty}\big(\Omega,C( [t_0,T];\mathbb{R}^{n\times n})\big)\big),\ i=1,2$.

(iv) There exists a constant $C$ such that
$\big|f(t,x,y,z,\kappa,u_1,\mu_1,\nu_1,\lambda_1,u_2,\mu_2,\nu_2,\lambda_2)\big|\le C\big(|x|\\+|y|+|z|+|\kappa|+|u_1|+|\mu_1|+|\nu_1|+|\lambda_1|+|u_2|+|\mu_2|+|\nu_2|+|\lambda_2|\big)$ for any $x,y,z,\kappa, u_1,\mu_1,\nu_1,\lambda_1,u_2,\mu_2,\\\nu_2,\lambda_2$, where $f=\hat{b},\hat{\sigma},l_i,h_i,\ i=1,2$.

(v) $\hat{b},\hat{\sigma},l_i$ are $\mathbb{F}$-progressively measurable for any $x,y,z,\kappa,u_1,\mu_1,\nu_1,\lambda_1,u_2,\mu_2,\nu_2,\lambda_2$, and $h_i$ is $\mathcal{F}_T$-measurable for any $x,y,z,\kappa,\ i=1,2$.

Under (\textbf{H6.1}), the SDDE (\ref{state equation NZG}) admits unique solution by Proposition \ref{themSSDE}. In addition, the cost functional (\ref{cost NZG}) is well-defined and \textbf{Problem (NZG)} is meaningful.

Assume that $(u_1^\ast(\cdot),u_2^\ast(\cdot))$ is a Nash equilibrium point, let $\hat{x}^\ast(\cdot)$ represents the corresponding state trajectory. Combining with the convexity of control domain, define the variational control 
\begin{equation*}
	u_i^\rho(\cdot):=u_i^\ast(\cdot)+\rho v_i(\cdot),\ 0\le\rho\le1,
\end{equation*}
where $v_i(\cdot)\in\mathcal{U}_i[t_0,T]$ such that $v_i(\cdot)+u_i^\ast(\cdot)\in\mathcal{U}_i[t_0,T],\ i=1,2$, and the corresponding state trajectory denoted as $\hat{x}^\rho(\cdot)$. Similar to the treatment of stochastic maximum principle in section 3. The variational equation is introduced as the following 
\begin{equation}\left\{\begin{aligned}\label{variational equation NZG}
		d\hat{x}_i(t)&=\Big[\hat{b}_x(t)\hat{x}_i(t)+\hat{b}_y(t)\hat{y}_i(t)+\hat{b}_z(t)\hat{z}_i(t)+\hat{b}_{\kappa}(t)\hat{\kappa}_i(t)\\
		&\qquad+\hat{b}_{u_i}(t)v_i(t)+\hat{b}_{\mu_i}(t)v_i(t-\delta)+\hat{b}_{\nu_i}(t)\int_{t_0}^t\phi_{2i}(t,s)v_i(s)ds\\
		&\qquad+\hat{b}_{\lambda_i}(t)\int_{t_0}^t\psi_{2i}(t,s)v_i(s)dW(s)\Big]dt\\
		&\quad+\Big[\hat{\sigma}_x(t)\hat{x}_i(t)+\hat{\sigma}_y(t)\hat{y}_i(t)+\hat{\sigma}_z(t)\hat{z}_i(t)+\hat{\sigma}_{\kappa}(t)\hat{\kappa}_i(t)\\
		&\qquad+\hat{\sigma}_{u_i}(t)v_i(t)+\hat{\sigma}_{\mu_i}(t)v_i(t-\delta)+\hat{\sigma}_{\nu_i}(t)\int_{t_0}^t\phi_{2i}(t,s)v_i(s)ds\\
		&\qquad+\hat{\sigma}_{\lambda_i}(t)\int_{t_0}^t\psi_{2i}(t,s)v_i(s)dW(s)\Big]dW(t),\quad t\in[t_0,T],\\
		\hat{x}(t)&=0,\ u_i(t)=\hat{\varsigma}_i(t),\quad t\in[t_0-\delta,t_0],\ i=1,2,
	\end{aligned}\right.\end{equation}
and the variational inequalities are
\begin{equation}\begin{aligned}\label{variational inequality NZG}
		\mathbb{E}\bigg\{&\int_{t_0}^T\Big[l_{ix}(t)\hat{x}_i(t)+l_{iy}(t)\hat{y}_i(t)+l_{iz}(t)\hat{z}_i(t)+l_{i\kappa}(t)\hat{\kappa}_i(t)+l_{iu}(t)v_i(t)+l_{i\mu}(t)v_i(t-\delta)\\
		&\qquad+l_{i\nu}(t)\int_{t_0}^t\phi_{2i}(t,s)v_i(s)ds+l_{i\lambda}(t)\int_{t_0}^t\psi_{2i}(t,s)v_i(s)dW(s) \Big]dt\\
		&\qquad+h_{ix}(T)\hat{x}_i(T)+h_{iy}(T)\hat{y}_i(T)+h_{iz}(T)\hat{z}_i(T)+h_{i\kappa}(T)\hat{\kappa}_i(T) \bigg\}\le0,\ i=1,2.
\end{aligned}\end{equation}
We affirm that the SDDE (\ref{variational equation NZG}) admits a unique solution under \textbf{(H6.1)} by Proposition \ref{themSSDE}.

Under \textbf{(H6.1)},  we can gain $\lim_{\rho\to 0}\mathbb{E}\big[\sup_{t\in[t_0,T]}|\check{X}_i(t)|^2\big]=0$, where $\check{X}_i(t):=\frac{\hat{x}^{\rho}(t)-\hat{x}^{\ast}(t)}{\rho}-\hat{x}_i(t),\ t\in[t_0,T],\ i=1,2$. This proof is similar to Lemma \ref{lem3.1}, thus, we omit it here.

Define $X_i(t)^\top:=\begin{bmatrix}
	\hat{x}_i(t)^\top & \hat{y}_i(t)^\top & \hat{z}_i(t)^\top & \hat{\kappa}_i(t)^\top
\end{bmatrix},\ i=1,2$, then we introduce the following adjoint equation 
\begin{equation}\left\{\begin{aligned}\label{adjoint equation NZG}
		(a)\ \eta_i(t)&=\mathbb{H}_i^\top-\int_t^T\zeta_i(s)dW(s),\quad t\in[t_0,T],\\
		(b)\ Y_i(t)&=\mathbb{L}_i(t)^\top+\bar{\mathbb{A}}_1(T,t)^\top\mathbb{H}_i ^\top+\bar{\mathbb{C}}_1(T,t)^\top \zeta_i(t)+\int_t^T\Big\{\check{\mathbb{A}}_2(t)^\top\psi_{1}(T,s)^\top\\
		&\quad\times\mathbb{E}^{\mathcal{F}_s}\big[D_s\mathbb{H}_i^\top\big]+\check{\mathbb{C}}_2(t)^\top \mathbb{E}^{\mathcal{F}_t}\big[D_t\big(\psi_{1}(T,s)^\top\zeta_i(s)\big)\big]\Big\}ds\\
		&\quad+\int_t^T\bigg\{\bar{\mathbb{A}}_1(s,t)^\top Y_i(s)+\bar{\mathbb{C}}_1(s,t)^\top Z_i(s,t)+\int_t^s\Big\{\check{\mathbb{A}}_2(t)^\top \psi_{1}(s,r)^\top\\
		&\quad\times\mathbb{E}^{\mathcal{F}_r}\big[D_rY_i(s)\big]+\check{\mathbb{C}}_2(t)^\top \mathbb{E}^{\mathcal{F}_t}\big[D_t\big(\psi_{1}(s,r)^\top Z_i(s,r)\big)\big]\Big\}dr\bigg\}ds\\
		&\quad-\int_t^TZ_i(t,s)dW(s),\quad t\in[t_0,T],\\
		(c)\ Y_i(t)&=\mathbb{E}^{\mathcal{F}_{t_0}}\big[Y(t)\big]+\int_{t_0}^tZ_i(t,s)dW(s),\quad t\in[t_0,T],\ i=1,2.
	\end{aligned}\right.\end{equation}
\normalsize 
where $\mathbb{H}_i,\mathbb{L}_i, \bar{\mathbb{A}}_1, \bar{\mathbb{C}}_1, \check{\mathbb{A}}_2, \check{\mathbb{C}}_2,\ i=1,2$ are defined similar to $\mathbb{H}, \mathbb{L}, \mathbb{A}_1, \mathbb{C}_1, \hat{\mathbb{A}}_2, \hat{\mathbb{C}}_2$ in (\ref{coefficient1}) and (\ref{cost coefficient3}) respectively, and just need to replace $b, \sigma, l, h$ by $\hat{b}, \hat{\sigma}, l_i, h_i,\ i=1,2$, respectively. In addition, we assume that the above equation (\ref{adjoint equation NZG}) admits unique solution $\big(\eta_i(\cdot),\zeta_i(\cdot),Y_i(\cdot),Z_i(\cdot,\cdot)\big),\ i=1,2$ as (\ref{adjoint equation}).  Furthermore, for $i=1,2$, let
\begin{equation}\left\{\begin{aligned}\label{defpqNZG}
		p_i(t):&=\eta_i^1(t)+\eta_i^2(t)\mathbb{I}_{[t_0,T-\delta)}(t)+\mathcal{E}_{1}(T,t)^\top\eta_i^3(t)+\mathbb{E}^{\mathcal{F}_t}\Big[\int_t^T\psi_{1}(T,r)^\top D_r\eta_i^4(r)dr\Big]\\
		&\quad+\mathbb{E}^{\mathcal{F}_t}\Big[\int_t^T Y_i^1(s)ds+\int_{t+\delta}^T Y_i^2(s)ds\mathbb{I}_{[t_0,T-\delta)}(t)+\int_t^T \mathcal{E}_{1}(s,t)^\top Y_i^3(s)ds\\
		&\quad+\int_t^T\int_t^s\psi_{1}(s,r)^\top D_rY_i^4(s)drds\Big],\\
		q(t):&=\zeta_i^1(t)+\zeta_i^2(t)\mathbb{I}_{[t_0,T-\delta)}(t)+\mathcal{E}_{1}(T,t)^\top\zeta_i^3(t)+\int_t^T\mathbb{E}^{\mathcal{F}_t}\big[D_t\big(\psi_{1}(T,r)^\top\zeta_i^4(r)\big)\big]dr\\
		&\quad+\int_t^T Z_i^1(s,t)ds+\int_{t+\delta}^T Z_i^2(s,t)ds\mathbb{I}_{[t_0,T-\delta)}(t)+\int_t^T \mathcal{E}_{1}(s,t)^\top Z_i^3(s,t)ds\\
		&\quad+\int_t^T\int_t^s\mathbb{E}^{\mathcal{F}_t}\big[D_t\big(\psi_{1}(s,r)^\top Z_i^4(s,r)\big)\big]drds.\\
	\end{aligned}\right.\end{equation}

Then, we define the Hamiltonian function as the following 
\begin{equation}\begin{aligned}
		&\mathcal{H}_i(t,x,y,z,\kappa,p_i,q_i,u_1,\mu_1,\nu_1,\lambda_1,u_2,\mu_2,\nu_2,\lambda_2)\\
		:=&\ l_i(t,x,y,z,\kappa,u_1,\mu_1,\nu_1,\lambda_1,u_2,\mu_2,\nu_2,\lambda_2)\\
		&+\big<\hat{b}(t,x,y,z,\kappa,u_1,\mu_1,\nu_1,\lambda_1,u_2,\mu_2,\nu_2,\lambda_2), p_i\big>\\
		&+\big<\hat{\sigma}(t,x,y,z,\kappa,u_1,\mu_1,\nu_1,\lambda_1,u_2,\mu_2,\nu_2,\lambda_2),q_i\big>. 	
\end{aligned}\end{equation}

\begin{mythm}\label{SPM-NZG}
	(Stochastic maximum principle) Let \textbf{(H6.1)} be satisfied. Suppose $(u_1^\ast(\cdot) ,u_2^\ast(\cdot))$ is a Nash equilibrium point to \textbf{Problem (NZG)} and $\hat{x}^\ast(\cdot)$ represents the corresponding optimal trajectory, $(\eta_i(\cdot),\zeta_i(\cdot),Y_i(\cdot),Z_i(\cdot,\cdot))$ is a solution to (\ref{adjoint equation NZG}), $(p_i(\cdot),q_i(\cdot))$ are defined by (\ref{defpqNZG}). Then, 
	\small 
	\begin{equation}\begin{aligned}\label{maximum principle NZG}
			&\bigg<\mathcal{H}^\ast_{u_i}(t)+\mathbb{E}^{\mathcal{F}_t}\Big[\mathcal{H}^\ast_{\mu_i}(t+\delta)\mathbb{I}_{[t_0,T-\delta)}(t)+\int_t^T\phi_{2i}(s,t)^\top \mathcal{H}^\ast_{\nu_i}(s)ds\\
			&\qquad +\int_t^T\psi_{2i}(s,t)^\top D_t\mathcal{H}^\ast_{\lambda_i}(s)ds \Big],u_i-u_i^{\ast}(t)\bigg>\le 0,\quad \forall u_i\in U_i,\ a.e.\ t\in[t_0,T],\ \mathbb{P}\mbox{-}a.s.,\ i=1,2.
	\end{aligned}\end{equation}
\end{mythm}	
\normalsize 

(\textbf{H6.2})
The Hamiltonian function $\mathcal{H}_i$ is concave in $(x,y,z,\kappa,u_1,\mu_1,\nu_1,\lambda_1,u_2,\mu_2,\nu_2,\lambda_2)$ and the terminal cost $h_i(x,y,z,\kappa)$ is concave in $(x,y,z,\kappa)$, for given $p_i(\cdot),q_i(\cdot),\ i=1,2$.

\begin{mythm}\label{sufficient condition NZG}
	Assume that $(\bar{u}_1(\cdot),\bar{u}_2(\cdot))\in\mathcal{U}_1[t_0,T]\times\mathcal{U}_2[t_0,T]$, $\bar{x}(\cdot)$ denotes the corresponding trajectory, $(p_i(\cdot),q_i(\cdot))$ defined by (\ref{defpqNZG}), where $(\eta_i(\cdot),\zeta_i(\cdot),Y_i(\cdot),Z_i(\cdot,\cdot)),\ i=1,2$ is a solution to (\ref{adjoint equation NZG}). Suppose that (\ref{maximum principle NZG}) is satisfied under the assumptions \textbf{(H6.1)} and \textbf{(H6.2)}, then $(\bar{u}_1(\cdot),\bar{u}_2(\cdot))$ is a Nash equilibrium point of \textbf{Problem (NZG)}.   
\end{mythm}
The result of Theorem \ref{SPM-NZG} (resp., Theorem \ref{sufficient condition NZG}) can be obtained by using the proof method of Theorem \ref{SPM} (resp., Theorem \ref{sufficient conition}), so we omit it here.

\subsection{A solvable LQ example}

In this section, we present an LQ example, to illustrate the stochastic maximum principle obtained in the previous section. For simplicity, we consider $n=1$.

The linear controlled stochastic system with extended mixed delay is considered as follows 
\begin{equation}\left\{\begin{aligned}\label{state equation LQ}
		dx(t)=&\big[f(t)u(t)+g(t)\mu(t)+h(t)\nu(t)+k(t)\lambda(t)\big]dt\\
		&+\big[\bar{a}(t)x(t)+\bar{b}(t)y(t)+\bar{c}(t)z(t)+\bar{d}(t)\kappa(t)+\bar{f}(t)u(t)\\
		&~\quad+\bar{g}(t)\mu(t)+\bar{h}(t)\nu(t)+\bar{k}(t)\lambda(t)\big]dW(t),\quad t\in[t_0,T],\\
		x(t)=&\ \xi(t), u(t)=\varsigma(t),\quad t\in[t_0-\delta,t_0],
	\end{aligned}\right.\end{equation}
with the following quadratic cost functional 
\begin{equation}\label{cost functional LQ}
	J\big(u(\cdot)\big)=\frac{1}{2}\mathbb{E}\bigg\{\int_{t_0}^T\Big[r_1(t)|u(t)|^2+r_2(t)|u(t-\delta)|^2\Big]dt+x(T)\bigg\},
\end{equation}
where $f(\cdot),g(\cdot),h(\cdot),k(\cdot),\bar{a}(\cdot),\bar{b}(\cdot),\bar{c}(\cdot),\bar{d}(\cdot),\bar{f}(\cdot),\bar{g}(\cdot),\bar{h}(\cdot),\bar{k}(\cdot)\in L_{\mathbb{F}}^{\infty}(\Omega,[t_0,T];\mathbb{R})$, $r_1(\cdot),r_2(\cdot)\in L^{\infty}(t_0,T;\mathbb{R}^+)$, $\phi_2(\cdot,\cdot),\psi_2(\cdot,\cdot)\in L^{\infty}\big([t_0,T];L^{\infty}\big(C([t_0,T];\mathbb{R})\big)\big)$.

Let $\mathcal{U}_{ad}$ be the admissible control set defined in (\ref{admissible control set}), with $m=1$ and $U=\mathbb{R}$ being a convex set. We seek an optimal control $u(\cdot)\in\mathcal{U}_{ad}$ that minimizes the cost functional (\ref{cost functional LQ}) subject to the state equation (\ref{state equation LQ}).

Now, the unified form of the adjoint equation (\ref{ABSDE}) now becomes
\begin{equation}\left\{\begin{aligned}
		p(t)=&\ p(T-\delta)+\int_t^T\Big\{\bar{a}(s)q(s)+\mathcal{E}_1(s,t)\bar{c}(s)q(s)+\int_s^t\psi_1(s,r)D_r\big[\bar{d}(s)q(s)\big]\\
		&+\mathbb{E}^{\mathcal{F}_s}\big[\bar{b}(s+\delta)q(s+\delta) \big]\Big\}ds-\int_t^Tq(s)dW(s),\quad t\in[t_0,T],\\
		p(T)=&\ 1,\ p(t)=0,\ t\in(T,T+\delta],\ q(t)=0,\quad t\in[T,T+\delta],  
	\end{aligned}\right.\end{equation}
which admits a unique $\mathcal{F}_t$-adapted solution $\big(p(\cdot),q(\cdot)\big)$. It is obviously for $t\in[0,T+\delta]$,
\begin{equation}
	p(t)\equiv1,\ q(t)\equiv0
\end{equation}
is the unique solution. In addition, the Hamilton function (\ref{hamilton}) is now derived as
\begin{equation}\begin{aligned}
		&\mathcal{H}(t,x,y,z,\kappa,p,q,u,\mu,\nu,\lambda)\\
		=&\ \big[\bar{a}(t)x+\bar{b}(t)y+\bar{c}(t)z+\bar{d}(t)\kappa\big]q+\frac{1}{2}\big[r_1(t)|u|^2+r_2(t)|\mu|^2\big]+\big[f(t)p+\bar{f}(t)q\big]u\\
		&+\big[g(t)p+\bar{g}(t)q\big]\mu+\big[h(t)p+\bar{h}(t)q\big]\nu+\big[k(t)p+\bar{k}(t)q\big]\lambda.
\end{aligned}\end{equation}
According to Theorem \ref{SPM}, one has 
\begin{equation}\begin{aligned}\label{LQSMPcondition}
		&\max_{u\in\mathcal{U}_{ad}}\Big<r_1(t)u^\ast(t)+f(t)+\mathbb{E}^{\mathcal{F}_t}\Big\{\big[r_2(t+\delta)u^\ast(t)+g(t+\delta)\big]\mathbb{I}_{[t_0,T-\delta)}(t)\\
		&\qquad\quad +\int_t^T\phi_2(s,t)h(s)ds+\int_t^T\psi_2(s,t)D_tk(s)ds\Big\},u-u^\ast(t)\Big>\\
		=&\ \Big<r_1(t)u^\ast(t)+f(t)+\mathbb{E}^{\mathcal{F}_t}\Big\{\big[r_2(t+\delta)u^\ast(t)+g(t+\delta)\big]\mathbb{I}_{[t_0,T-\delta)}(t)\\
		&\qquad\quad +\int_t^T\phi_2(s,t)h(s)ds+\int_t^T\psi_2(s,t)D_tk(s)ds\Big\},u^\ast(t)\Big>\\
		&\qquad\qquad\qquad\qquad\qquad\qquad\qquad\forall u\in U,\ a.e.t\in[t_0,T],\ \mathbb{P}\mbox{-}a.s., 	
\end{aligned}\end{equation}
where $(u^\ast(\cdot),x^\ast(\cdot))$ is the optimal pair. From (\ref{LQSMPcondition}), it is easy to gain that the optimal control is the form as
\begin{equation}\begin{aligned}\label{optimal control LQ}
		u^\ast(t)=-\frac{f(t)+\mathbb{E}^{\mathcal{F}_t}\Big[g(t+\delta)\mathbb{I}_{[t_0,T-\delta)}(t)+\int_t^T\phi_2(s,t)h(s)ds+\int_t^T\psi_2(s,t)D_tk(s)ds \Big]}{2\big[r_1(t)+r_2(t+\delta)\big]},&\\
		a.e.\ t\in[t_0,T].&
\end{aligned}\end{equation}
Especially, if we let $f(\cdot)\equiv g(\cdot)\equiv \phi(\cdot,\cdot)\equiv h(\cdot)\equiv k(\cdot)\equiv2$, $r_1(\cdot)\equiv r_1(\cdot)\equiv1$, then from (\ref{optimal control LQ}) we know that $u^\ast(t)=T-t+1,\ a.e.\ t\in[t_0,T]$. Noting $U=\mathbb{R}$ is convex, this optimal control can  be adopted.

\section{Concluding remarks}

In this paper, we have studied a class of stochastic optimal control problems with extended mixed delays, where include point delay, extended distributed delay and extended noisy memory, and the control domain is a convex set. It is worth mentioning that the extended noisy memory of control is proposed for the first time in this paper, and the extended mixed delay of state variables and control variables not only enters the drift term and diffusion term of the state equation, but also are component of the running cost and terminal term.

With the help of stochastic Fubini's theorem, the variational equation with extended mixed delay is successfully transformed into an SVIE without delay. Furthermore, a class of BSVIEs with Malliavin derivatives is introduced as the adjoint equation by using coefficient decomposition and generalized duality principle, and the maximum principle and verification theorem are established. By exploiting Clark-Ocone's formula, the adjoint equation is reformulated as a class of ABSDEs containing Malliavin derivatives. As applications, a class of nonzero-sum stochastic differential games with extended mixed delay and a class of LQ problem are studied. 

However, in this paper, we only assume that the solution of BSVIE (\ref{adjoint equation}) containing Malliavin derivatives exists. In the future, we will consider establishing the existence and uniqueness of adapted $M$-solutions to this kind of equations and prove the general maximum principle.

\end{document}